\pdfoutput=1
\documentclass[a4paper]{amsart}

\title{Supersymmetric field theories from twisted vector bundles}
\author{Augusto Stoffel}
\address{
  Max Planck Institute for Mathematics\\
  Vivatsgasse 7\\
  53111 Bonn\\
  Germany}
\email{astoffel@mpim-bonn.mpg.de}

\usepackage[utf8]{inputenc}
\usepackage[T1]{fontenc}

\usepackage{amsmath,amsthm,amssymb}
\usepackage{microtype}

\usepackage{hyperref}
\hypersetup{
  pdftitle={\csname @title\endcsname},
  pdfauthor={\authors},
  hidelinks}

\usepackage[hyperref,safeinputenc,maxbibnames=99]{biblatex}
\usepackage[all,cmtip]{xy}

\xymatrixrowsep={1.5em}
\xymatrixcolsep={2.4em}
\objectmargin={0.9ex}

\usepackage{tikz}
\usetikzlibrary{quotes}
\tikzset{
    l/.style={xshift=-4pt},
    r/.style={xshift=4pt},
    L/.style={out=180, in=180},
    R/.style={out=0, in=0},
    S/.style={out=0, in=180},
    dotp/.pic={\fill (0,0) circle[radius=1pt];},
    dotm/.pic={\draw (0,0) circle[radius=1pt];},
    cross/.style={preaction={draw,ultra thick,white}},
}

\numberwithin{equation}{section}
\newtheorem{theorem}[equation]{Theorem}
\newtheorem{proposition}[equation]{Proposition}
\newtheorem{lemma}[equation]{Lemma}
\newtheorem{corollary}[equation]{Corollary}

\theoremstyle{definition}

\theoremstyle{remark}

\newtheorem{remark}[equation]{Remark}

\newcommand\abs[1]{\lvert #1\rvert}

\newcommand\Diff{\operatorname{Diff}}
\newcommand\End{\operatorname{End}}
\newcommand\Fun{\operatorname{Fun}}
\newcommand\Hom{\operatorname{Hom}}
\newcommand\Id{\operatorname{Id}}
\newcommand\Isom{\operatorname{Isom}}

\newcommand\holim{\operatorname*{holim}}
\newcommand\id{\operatorname{id}}

\newcommand\pr{\operatorname{pr}}
\newcommand\str{\operatorname{str}}

\newcommand\vol{\operatorname{vol}}

\def\mystack#1{\mathrel{\vcenter{\offinterlineskip\ialign{$##$\cr#1\cr}}}}
\newcommand\threerightarrows{\mystack{\to\cr\to\cr\to}}

\newcommand\EFT{\mathhyphen\mathrm{EFT}}
\newcommand\TFT{\mathhyphen\mathrm{TFT}}
\newcommand\ETw{\mathhyphen\mathrm{ETw}}

\newcommand\Bord{\mathhyphen\mathrm{Bord}}
\newcommand\EBord{\mathhyphen\mathrm{EBord}}
\newcommand\Vect{\mathrm{Vect}}
\newcommand\SM{\mathrm{SM}}
\newcommand\Man{\mathrm{Man}}

\newcommand\mathhyphen{\textrm{-}}
\newcommand\mathemdash{\textrm{---}}

\newcommand\sslash{/\kern-0.7ex/}
\newcommand\backsslash{\backslash\kern-0.7ex\backslash}

\begin{document}

\begin{abstract}
  We give a description of the delocalized twisted cohomology of an
  orbifold and the Chern character of a twisted vector bundle in terms
  of supersymmetric Euclidean field theories.  This includes the
  construction of a twist functor for $1\vert1$-dimensional EFTs from
  the data of a gerbe with connection.
\end{abstract}

\maketitle
  
\section{Introduction}
\label{sec:introduction}

In this paper, we explore a twisted version of the Stolz--Teichner
program on the use of supersymmetric Euclidean field theories (EFTs)
as geometric cocycles for cohomology theories \cite{MR2742432}.  We
focus on twisted $1\vert1$ and $0\vert1$-dimensional EFTs over an
orbifold $\mathfrak X$; the corresponding cohomology theories are
(twisted) $K$-theory and delocalized de Rham cohomology.  One of our
main goals is to describe the Chern character
\begin{equation*}
  \mathrm{ch}\colon K^\alpha(\mathfrak X)
  \to H^{\mathrm{ev}}_{\mathrm{deloc}}(\mathfrak X,\alpha)
\end{equation*}
of a twisted vector bundle in terms of dimensional reduction of field
theories.  (For compact $\mathfrak X$, this Chern character map
provides an isomorphism after complexification; thus, delocalized
cohomology, which we recall below, is a stronger invariant than
regular de Rham cohomology.)  Here, the twist is $\alpha \in
H^3(\mathfrak X;\mathbb Z)$.  Thus, on the field theory side, our
first task is to construct from $\alpha$ a Euclidean twist functor (or
anomaly)
\begin{equation*}
  T \in 1\vert1\ETw(\mathfrak X)
\end{equation*}
for $1\vert1$-EFTs over $\mathfrak X$ and describe its dimensional
reduction $T' \in 0\vert1\ETw(\Lambda\mathfrak X)$, which is a twist
over the inertia orbifold $\Lambda\mathfrak X$.  It will turn out, as
expected, that $T'$-twisted field theories model the delocalized
cohomology group $H^{\mathrm{ev}}_{\mathrm{deloc}}(\mathfrak
X,\alpha)$.  Next, we construct, from the data of an $\alpha$-twisted
vector bundle $\mathfrak V$ on $\mathfrak X$, a $T$-twisted
$1\vert1$-EFT, and show that its dimensional reduction, which is
$T'$-twisted, corresponds to $\mathrm{ch}(\mathfrak V)$.

\subsection{Field theories and twisted cohomology}
\label{sec:field-theor-orbif}

In this paper, we use the Stolz--Teichner framework of geometric field
theories laid out in \cite{MR2742432}, which draws on the functorial
approach to quantum field theory of Segal, Atiyah and many others.  A
supersymmetric Euclidean (quantum) field theory of dimension
$d\vert\delta$ over an orbifold $\mathfrak X$ is a symmetric monoidal
functor
\begin{equation*}
  E \in \Fun^\otimes_\SM(d\vert\delta\EBord(\mathfrak X), \Vect)
\end{equation*}
between a bordism category and the category $\Vect$ of complex super
vector spaces.  Roughly speaking, the bordism category in question has
closed $(d-1)\vert\delta$-dimensional supermanifolds as objects and
$d\vert\delta$-dimensional bordism between them as morphisms; all
supermanifolds are equipped with a Euclidean structure (which boils
down to a flat Riemannian metric in the purely bosonic case $\delta =
0$) and a smooth map to $\mathfrak X$.  Thus, $E$ can be thought of as
a family of field theories parametrized by $\mathfrak X$.  Field
theories can be pulled back along maps $\mathfrak Y \to \mathfrak X$.
The subscript ``$\SM$'' above indicates that we require the assignment
$E$ to be smooth, in the sense that it sends smooth families of
objects and morphisms to smooth families.  To make precise sense of
this, we promote $d\vert\delta\EBord(X)$ and $\Vect$ to internal
categories in symmetric monoidal stacks over the site $\SM$ of
supermanifolds, and $E$ to a functor of internal categories.

Many interesting constructions do not quite produce a field theory as
defined above, but rather an ``anomalous'' or twisted theory
\cite[][etc.]{MR3330283, MR3355809, MR2079383}.  In our framework, those are
defined as follows.  We write
\begin{equation*}
  d\vert\delta\ETw(\mathfrak X)
   = \Fun^\otimes_\SM(d\vert\delta\EBord(\mathfrak X), \mathrm{Alg})
\end{equation*}
for the groupoid of $d\vert\delta$-dimensional Euclidean twists over
$\mathfrak X$.  Here $\mathrm{Alg}$ is the internal category of
(bundles of) algebras, bimodules, and bimodule maps.  Finally, given
$T \in d\vert\delta\ETw(\mathfrak X)$, a twisted field theory is a
natural transformation $E$
\begin{equation*}
  \xymatrix@C=5em{
    d\vert\delta\EBord(\mathfrak X)
    \ar@/^{4ex}/[r]^1_{}="a"
    \ar@/_{4ex}/[r]_T^{}="b"
    \ar^E@{=>}"a";"b"
    & \mathrm{Alg}}
\end{equation*}
from the trivial twist $1$ (which maps everything to $\mathbb C$) to
$T$.  We write
\begin{math}
  d\vert\delta\EBord^T(\mathfrak X)
\end{math}
for the groupoid of $T$-twisted Euclidean field theories over
$\mathfrak X$. (See \cite{MR2742432} for the complete definitions,
including more details on the categorical and supergeometry aspects.)

A conjecture of Stolz and Teichner \cite{MR2079378} states that
$2\vert1$-EFTs provide geometric cocycles for the cohomology theory
$\mathrm{TMF}$ of topological modular forms, in the sense that, for
any manifold $X$,
\begin{equation*}
  2\vert1\EFT^n(X)/\text{concordance} \cong \mathrm{TMF}^n(X).
\end{equation*}
Here, two field theories $E_0,E_1$ are said to be concordant if there
exists $E\in d\vert\delta\EFT(X \times \mathbb R)$ such that $E_i
\cong E\vert_{X\times\{ i \}}$.  (Among other difficulties, a solution
to the conjecture certainly requires that we refine the definitions
and pass to fully extended geometric field theories.)  In this paper,
we focus on the $1\vert1$ and $0\vert1$-dimensional cases, where an
analogue of that conjecture states that the relevant cohomology
theories are topological $K$-theory and de Rham cohomology
\cite{MR2648897, MR2763085}.  When we replace the background manifold
$X$ by an orbifold $\mathfrak X$, it is natural to ask what kind of
information about twisted equivariant cohomology such field theories
capture---but the orbifold perspective is important to deal with
twists, even if $\mathfrak X$ is equivalent to a manifold.

We begin our study with a classification of $0\vert1$-dimensional
twists for EFTs over an orbifold (section~\ref{sec:twisted-de-rham}).
For a particular $T_\alpha \in 0\vert1\ETw(\Lambda\mathfrak X)$,
concordance classes of twisted EFTs over the inertia $\Lambda\mathfrak
X$ (the orbifold of ``constant loops'') are in natural bijection with
the delocalized twisted cohomology $H^*_{\mathrm{deloc}}(\mathfrak X,
\alpha)$.  Then, turning to $1\vert1$-dimensional considerations, we
construct $T_{\tilde{\mathfrak X}} \in 1\vert1\ETw(\mathfrak X)$
taking as input $\alpha \in H^3(\mathfrak X;\mathbb Z)$, or, rather, a
($\mathbb C^\times$-)gerbe with connection $\tilde{\mathfrak X} \to
\mathfrak X$ representing that class (section~\ref{sec:twists}).  This
is an extension of the transgression construction for gerbes
\cite{MR2362847} in the sense that it produces, in particular, a line
bundle on the stack $\mathfrak K(\mathfrak X)$ of supercircles over
$\mathfrak X$; that stack is a super analogue of $L\mathfrak X\sslash
\Diff^+(S^1)$, and the line bundle we obtain is a super analogue of
the usual transgression of the gerbe.

It is now reasonable to conjecture that
\begin{equation}
  \label{eq:18}
  1\vert1\EFT^{T_{\tilde{\mathfrak X}}}(\mathfrak X)/\mathrm{concordance}
 \cong K^\alpha(\mathfrak X),
\end{equation}
but this question is open even in the case where $\mathfrak X$ is a
manifold and $T_{\tilde{\mathfrak X}}$ is trivial, so we will not
dwell on it here.  Instead, we will demonstrate the meaningfulness of
our construction by associating a twisted field theory $E_{\mathfrak
  V}$ to any $\tilde{\mathfrak X}$-twisted vector bundle $\mathfrak
V$, and identifying its dimensional reduction.  Here again, the
partition function of $E_{\mathfrak V}$ is the super counterpart of a
classical construction, namely the trace of the holonomy, which in
this case is not a function but rather a section of the transgression
of $\tilde{\mathfrak X}$.

\subsection{Dimensional reduction and the Chern character}
\label{sec:dimens-reduct-chern}

Dimensional reduction is, intuitively, the assignment of a
$(d-1)$-dimensional theory to a $d$-dimensional theory induced by the
functor of bordism categories $S^1\times\mathemdash \colon (d-1)\Bord
\to d\Bord$.  For field theories over an orbifold, the action of the
circle group $\mathbb T$ on the inertia $\Lambda\mathfrak X$ can be
used to refine this to (partial) assignments
\begin{equation*}
  1\vert1\ETw(\mathfrak X) \to 0\vert1\ETw(\Lambda\mathfrak X), \qquad
  1\vert1\EFT^T(\mathfrak X) \to 0\vert1\EFT^{T'}(\Lambda\mathfrak X).
\end{equation*}
Our dimensional reduction procedure was developed in
\cite{arXiv:1703.00314} and is recalled in
section~\ref{sec:remind-dimens-reduct}.  It is given by the pull-push
operation along functors between certain variants of the corresponding
Euclidean bordism categories
\begin{equation}
  \label{eq:21}
  0\vert1\EBord(\Lambda\mathfrak X)
  \leftarrow 0\vert1\EBord^{\mathbb T}(\Lambda\mathfrak X)
  \to 1\vert1\EBord(\mathfrak X).
\end{equation}
The lack of direct map from left to right is due to certain subtleties
concerning Euclidean supergeometry, as explained in the reference
above.

The results of this paper can be summarized in the following
statement.

\begin{theorem}
  \label{thm:4}
  Let $\mathfrak X$ be an orbifold.  To any gerbe with connection
  $\tilde{\mathfrak X}$ and $\tilde{\mathfrak X}$-twisted vector
  bundle $\mathfrak V$ over $\mathfrak X$, correspond
  \begin{equation*}
    T_{\tilde{\mathfrak X}} \in 1\vert1\ETw(\mathfrak X),
    \quad E_{\mathfrak V} \in 1\vert1\EFT^{T_{\tilde{\mathfrak X}}}(\mathfrak X)
  \end{equation*}
  such that the diagram
  \begin{equation*}
    \xymatrix@R=0pt{
      & 1\vert1\EFT^{T_{\tilde{\mathfrak X}}}(\mathfrak X)
      \ar[r]^-{\operatorname{red}}\ar@{.>}[dd]
      & 0\vert1\EFT^{T'_{\tilde{\mathfrak X}}}(\Lambda\mathfrak X) \ar[dd]
      \\
      \mathrm{Vect}^{\tilde{\mathfrak X}}(\mathfrak X) \ar@/^2ex/[ur]^E\ar@/_2ex/[dr]
      &&\\
      & K^{\alpha}(\mathfrak X) \ar[r]^-{\mathrm{ch}} 
      & H^{\mathrm{ev}}_{\mathrm{deloc}}(\mathfrak X, \alpha)}
  \end{equation*}
  commutes.  Here, $\alpha \in H^3(\mathfrak X; \mathbb Z)$ is the
  Dixmier--Douady class of $\tilde{\mathfrak X}$ and
  $T'_{\tilde{\mathfrak X}}$ is the dimensional reduction of
  $T_{\tilde{\mathfrak X}}$.
\end{theorem}

Since the right vertical map is a bijection on concordance classes
(theorem~\ref{thm:2}), this gives a geometric interpretation of the
twisted orbifold Chern character.

This theorem generalizes results of \textcite{arXiv:0711.3862} and
\textcite{arXiv:1202.2719} for the untwisted, non-equivariant case.
In a different direction, we point out that
\textcite{arXiv:1610.02362} extended that story to the equivariant
case for Lie group actions.  We also point out that, in the case of a
global quotient orbifold $X\sslash G$ and twist coming from a central
extension of $G$, \textcite{arXiv:1410.5500} obtained a somewhat
different field-theoretic interpretation of $K^\alpha(X\sslash G)
\otimes \mathbb C$.  In fact, he also found a description of
$\mathrm{TMF} \otimes \mathbb C$ in terms of ``simple''
$2\vert1$-EFTs, and it would be interesting to investigate if these
are obtained, in our language, as the dimensional reduction of
full-blown $2\vert1$-EFTs.

\subsection{Notation and conventions}
\label{sec:notation}

We generally follow Deligne and Morgan's \cite{MR1701597} treatment of
supermanifolds.  We also use Dumitrescu's \cite{MR2407109} notion of
super parallel transport (for connections), with the difference that
we always transport along the \emph{left}-invariant vector field $D =
\partial_\theta - \theta\partial_t$ of $\mathbb R^{1\vert1}$.  For the
notion of Euclidean structures (in dimension $1\vert1$ and $0\vert1$),
see \cite[appendix B]{arXiv:1703.00314}.

Most manipulations in this paper happen in the bicategory of stacks
(Grothendieck fibrations satisfying descent) over the site of
supermanifolds (see e.g.\ \textcite{MR2817778} for details).  Every Lie
groupoid presents a stack, and we use, concretely, the stack of
torsors as a model.  An orbifold is a stack presented by a proper
étale Lie groupoid.  We fix, once and for all, an étale Lie groupoid
presentation for our orbifold $\mathfrak X$,
\begin{equation*}
  s,t\colon X_1 \rightrightarrows X_0.
\end{equation*}
This determines presentations
\begin{equation*}
\hat X_1 \rightrightarrows \hat X_0, \quad
\Pi T X_1 \rightrightarrows \Pi T X_0, \quad
\Pi T\hat X_1 \rightrightarrows \Pi T\hat X_0,
\end{equation*}
of $\Lambda\mathfrak X$ (the inertia orbifold), $\Pi T\mathfrak X$
(the stack of maps $\mathbb R^{0\vert1} \to \mathfrak X$), and $\Pi
T\Lambda\mathfrak X$ respectively.  There are obvious maps
\begin{equation*}
  \mathfrak X
  \overset i\to \Lambda\mathfrak X
  \overset p\to \mathfrak X,
  \quad
  \Pi T\mathfrak X
  \overset i\to \Pi T\Lambda\mathfrak X
  \overset p\to \Pi T\mathfrak X
\end{equation*}
that we often leave implicit.

We fix also a gerbe with connection $\tilde{\mathfrak X} \to \mathfrak
X$ and, when needed, an $\tilde{\mathfrak X}$-twisted vector bundle
$\mathfrak V$; those are assumed to come with presentations as well,
with the notation introduced in appendix~\ref{sec:twisted-superconn}.
Note that the chosen presentation $X_1 \rightrightarrows X_0$ of
$\mathfrak X$ must be such that $\tilde{\mathfrak X}$ admits a
presentation as a central extension $L \to X_1$, which is a nontrivial
condition.  For instance, when $\mathfrak X$ is just a manifold, we
need, in general, to choose as presentation the Čech groupoid
$\coprod_{i,j} U_i \cap U_j \rightrightarrows \coprod_i U_i$ of some
open cover.

\subsection{Acknowledgments}
\label{sec:acknowledgments}

This paper is based on a part of my Ph.D. thesis \cite{MR3553617}, and
I would like to thank my advisor, Stephan Stolz, for the guidance.  I
would also like to thank Matthias Ludewig, Byungdo Park, Peter
Teichner, and Peter Ulrickson for valuable discussions, and Karsten
Grove for the financial support during my last semester as a graduate
student (NSF grant DMS-1209387).

\section{Twisted $0|1$-EFTs and de~Rham cohomology}
\label{sec:twisted-de-rham}

In this section, we extend in two directions the results of
\textcite{MR2763085} on the relation between $0$-dimensional
supersymmetric field theories over a manifold and de Rham cohomology.
First, we replace the target manifold by an orbifold, and, second, we
provide a classification of twists.  This provides, in particular, a
field-theoretic description of the delocalized twisted de Rham
cohomology of an orbifold, which is isomorphic, via the Chern
character, to complexified twisted $K$-theory \cite{MR1993337,
  MR2271013}.

We denote by $\mathfrak B(\mathfrak X)$ the stack of fiberwise
connected bordisms in $0\vert1\EBord(\mathfrak X)$, which can be
described concretely as
\begin{equation*}
  \mathfrak B(\mathfrak X) = \underline{\Fun}_{\mathrm{SM}}(\mathbb
  R^{0\vert1}, \mathfrak X)\sslash\mathrm{Isom}(\mathbb R^{0\vert1}).
\end{equation*}
When $\mathfrak X$ is a manifold, the mapping stack
$\underline{\mathrm{Fun}}_{\mathrm{SM}}(\mathbb R^{0\vert1}, \mathfrak
X)$ is represented by the parity-reversed tangent bundle, so we will in
general write $\Pi T\mathfrak X = \underline {\Fun}_{\mathrm{SM}}
(\mathbb R^{0\vert1}, \mathfrak X)$ for the stack of superpoints.  If
the stack $\mathfrak X$ admits a Lie groupoid presentation
$X_1\rightrightarrows X_0$, then $\Pi T\mathfrak X$ can be presented
by the Lie groupoid $\Pi TX_1 \rightrightarrows \Pi TX_0$; in
particular, if $\mathfrak X$ is an orbifold, $\Pi T \mathfrak X$ is
again an orbifold.

We define the groupoid of Euclidean $0\vert1$-twists over $\mathfrak
X$ to be
\begin{gather*}
  0\vert1\ETw(\mathfrak X) = \Fun_{\mathrm{SM}}(\mathfrak B(\mathfrak
  X), \mathrm{Vect})
\end{gather*}
and, for each $T \in 0\vert1\ETw(\mathfrak X)$, the corresponding set
of $T$-twisted topological respectively Euclidean field theories over
$\mathfrak X$ to be the set of global sections of $T$:
\begin{gather*}
  0\vert1\EFT^T(\mathfrak X) = C^\infty(\mathfrak B(\mathfrak X), T).
\end{gather*}
In these definitions, $\mathrm{Vect}$ can be the stack of real or
complex super vector bundles, but ultimately we are interested in the
complex case.

We recall the construction, in \textcite[definition 6.2]{MR2763085},
of the twist
\begin{equation*}
  T_1\colon \mathfrak B(\mathrm{pt}) = \mathrm{pt}\sslash \mathrm{Isom}(\mathbb R^{0\vert1}) \to \mathrm{Vect}.
\end{equation*}
This functor is entirely specified by the requirement that the point
$\mathrm{pt}$ maps to the odd complex line $\Pi\mathbb C$, and by a
group homomorphism $\mathrm{Isom}(\mathbb R^{0\vert1}) \to
\mathrm{GL}(0\vert1)\cong\mathbb C^{\times}$, which we take to be the
projection onto $\mathbb Z/2 = \{ \pm 1 \}$.  We set $T_n =
T_1^{\otimes n}$, and use the same notation for the pullback of those
line bundles to $\mathfrak B(\mathfrak X)$.

\subsection{Superconnections and twists}
\label{sec:superc-twists}

Following \textcite{MR790678}, we define a superconnection $\mathbb A$
on a $\mathbb Z/2$-graded complex vector bundle $V \to X$ to be an odd
operator (with respect to the total $\mathbb Z/2$-grading) on
$\Omega^*(X; V)$ satisfying the Leibniz rule
\begin{equation}
  \label{eq:12}
  \mathbb A(\omega f) = (d\omega) f + (-1)^{\abs\omega} \omega \mathbb
  A f.
\end{equation}
Here, $\omega \in \Omega^*(X)$ and $f \in \Omega^*(X; V)$.  It follows
that $\mathbb A$ is entirely determined by its restriction to
$\Omega^0(X;V)$; denoting by $A_i$, $i\geq 0$, the component
$\Omega^0(X;V) \to \Omega^i(X;V)$, we find that $A_1$ is an affine
(even) connection and all other $A_i$ are $\Omega^0(X)$-linear odd
homomorphisms.  The even operator $\mathbb A^2\colon \Omega^*(X;V) \to
\Omega^*(X;V)$ is $\Omega^*(X)$-linear, and is called the curvature of
$\mathbb A$.  In particular, a flat superconnection is a differential
on $\Omega^*(X; V)$.

Now, let $V_0, V_1 \to X$ be complex super vector bundles and $\mathbb
A_i$, $i=0,1$, superconnections.  Then there exists a superconnection
$\mathbb A$ on the homomorphism bundle $\Hom(V_0,V_1) \to X$,
characterized by
\begin{equation*}
  (\mathbb A\Phi)f = \mathbb A_1(\Phi f)
  - (-1)^{\abs{\Phi}}\Phi(\mathbb A_0 f)
\end{equation*}
for any section $\Phi$ of $\Omega^*(X;\Hom(V_0,V_1))$ of parity
$\abs{\Phi}$ and $f \in \Omega^*(X;V_0)$.  We define
$\mathrm{Vect}^{\mathbb A}$ to be the prestack on $\Man$ whose objects
over $X$ are vector bundles with superconnection $(V,\mathbb A)$, and
morphisms $(V_0,\mathbb A_0) \to (V_1,\mathbb A_1)$ are sections $\Phi
\in \Omega^*(X;\Hom(V_0,V_1))$ of even total degree satisfying
$\mathbb A(\Phi) = 0$.  This turns out to be a stack.

There is a nice interpretation of superconnections in terms of
Euclidean supergeometry.  Consider the pullback bundle $\pi^*V \to \Pi
TX$ along $\pi\colon \Pi TX \to X$.  Its sections on an open $U
\subset X$ are given by $\Omega^*(U) \otimes_{C^\infty(U)} C^\infty(U;
V) = \Omega(U; V)$, and to say that a given odd, fiberwise linear
vector field $\mathbb A$ on $\pi^*V$ is $\pi$-related to the de Rham
vector field $d$ on the base is precisely the same as saying that
equation \eqref{eq:12} holds.  Thus a superconnection on $V$ gives
$\pi^*V$ the structure of an $\Isom(\mathbb R^{1\vert1})$-equivariant
vector bundle over $\Pi TX$, where the action on the base is via the
projection $\Isom(\mathbb R^{1\vert1})\to\Isom(\mathbb R^{0\vert1})$
and the identification $\Pi TX = \underline{\mathrm{SM}}(\mathbb
R^{0\vert1}, X)$.  The superconnection is flat if and only if this
action factors through $\Isom(\mathbb R^{0\vert1})$.  There is also a
converse statement.

\begin{theorem}
  \label{thm:1}
  The stack map $\mathrm{Vect}^{\mathbb A} \to \mathrm{Vect}(\Pi
  T\textrm{---}\sslash \Isom(\mathbb R^{1\vert1}))$ defined above is
  an equivalence.  The same is true for the map
  $\mathrm{Vect}^{\mathbb A\flat} \to \mathrm{Vect}(\Pi
  T\textrm{---}\sslash \Isom(\mathbb R^{0\vert1}))$.
\end{theorem}

As usual, we extend the above definitions by saying that a vector
bundle with superconnection on a stack $\mathfrak X$ is a fibered
functor $V\colon \mathfrak X \to \mathrm{Vect}^{\mathbb A}$; it is
flat if it takes values in the substack $\mathrm{Vect}^{\mathbb
  A\flat}$ of flat superconnections.  With this in place, we can
return to our discussion of twisted field theories.

\begin{proposition}
  \label{prop:3}
  For $\mathfrak X$ a differentiable stack, there is a natural
  equivalence of groupoids
  \begin{equation*}
    \mathrm{Vect}^{\mathbb A\flat}(\mathfrak X) \to 0\vert1\ETw(\mathfrak X).
  \end{equation*}
\end{proposition}

\begin{proof}
  There exists a bisimplicial manifold $\{ \Pi TX_j \times
  \Isom(\mathbb R^{0\vert1})^{\times i} \}_{i,j\geq 0}$ whose vertical
  structure maps give nerves of Lie groupoids presenting $\Pi
  T\mathfrak X \times \Isom(\mathbb R^{0\vert1})^{\times i}$ and whose
  horizontal structure maps give nerves of presentations of $\Pi
  TX_j\sslash\Isom(\mathbb R^{0\vert1})$.  Applying $\mathrm{Vect}$,
  we get a double cosimplicial groupoid
  \begin{equation}
    \label{eq:19}
    \begin{gathered}
      \xymatrix@R=1em@C=1.2em{
        \vdots & \vdots\\
        \mathrm{Vect}(\Pi TX_1) \ar@<-.66ex>[r]\ar@<.66ex>[r]
        \ar@<-1.33ex>[u]\ar@<+1.33ex>[u]\ar[u] & \mathrm{Vect}(\Pi
        TX_1 \times \Isom(\mathbb R^{0\vert1}))
        \ar@<-1.33ex>[r]\ar@<+1.33ex>[r]\ar[r]\ar@<-1.33ex>[u]\ar@<+1.33ex>[u]\ar[u]
        & \cdots
        \\
        \mathrm{Vect}(\Pi TX_0) \ar@<-.66ex>[r]\ar@<.66ex>[r]
        \ar@<-.66ex>[u]\ar@<.66ex>[u] & \mathrm{Vect}(\Pi TX_0 \times
        \Isom(\mathbb R^{0\vert1}))
        \ar@<-1.33ex>[r]\ar@<+1.33ex>[r]\ar[r]\ar@<-.66ex>[u]\ar@<+.66ex>[u]
        & \cdots }
    \end{gathered}
  \end{equation}

  Now we calculate the (homotopy) limit of this diagram in two
  different ways.  Taking the limit of the columns and then the limit
  of the resulting cosimplicial groupoid, we get, by proposition~8 of
  \cite{arXiv:1703.00314},
  \begin{multline*}
  \holim \left(  \Fun_{\mathrm{SM}}(\Pi T\mathfrak X, \mathrm{Vect})
  \rightrightarrows \Fun_{\mathrm{SM}}(\Pi T\mathfrak X \times
  \Isom(\mathbb R^{0\vert1}), \mathrm{Vect})\threerightarrows \cdots
\right)\\
     \cong  \Fun_{\mathrm{SM}}(\Pi T\mathfrak X\sslash \Isom(\mathbb
     R^{0\vert1}), \mathrm{Vect}) = 0\vert1\ETw(\mathfrak X).
  \end{multline*}
  On the other hand the limit of each row is equivalent to
  $\mathrm{Vect}(\Pi TX_i\sslash \Isom(\mathbb R^{0\vert1}))$, and the
  stack map $\mathrm{Vect}^{\mathbb A\flat} \to \mathrm{Vect}(\Pi
  T\text{---}\sslash\Isom(\mathbb R^{0\vert1}))$ of
  theorem~\ref{thm:1} gives us a levelwise equivalence of simplicial
  groupoids
  \begin{equation*}
    \mathrm{Vect}^{\mathbb A\flat}(X_\bullet) \to \mathrm{Vect}(\Pi TX_\bullet\sslash \Isom(\mathbb
    R^{0\vert1})).
  \end{equation*}
  Taking limits, we get an equivalence
  \begin{math}
    \mathrm{Vect}^{\mathbb A\flat}(\mathfrak X) \to
    0\vert1\ETw(\mathfrak X).
  \end{math}
\end{proof}

\begin{remark}
  \label{rem:nonflat}
  We can drop the flatness condition by considering vector bundles on
  $\Pi T\mathfrak X\sslash\Isom(\mathbb R^{1\vert1})$.
\end{remark}

\subsection{Concordance of flat sections}
\label{sec:conc-flat-sect}

The goal of this section is to identify concordance classes of twisted
$0\vert1$-EFTs.  This is an extension of the well-known fact that
closed differential forms are concordant through closed forms if and
only if they are cohomologous; the extension takes place in two
orthogonal directions: we replace manifolds with differentiable stacks
and the trivial flat line bundle with an arbitrary flat
superconnection.  Fix a differentiable stack $\mathfrak X$ with
presentation $X_1\rightrightarrows X_0$ and let $T \in
0\vert1\ETw(\mathfrak X)$ be the twist associated to the flat
superconnection $(V,\mathbb A)$ on $\mathfrak X$.

\begin{proposition}
\label{prop:1}
  There are natural bijections
  \begin{gather*}
    0\vert1\EFT^T(\mathfrak X) \cong \{ \text{$\mathbb A$-closed even
      forms with values
      in } V \},\\
    0\vert1\EFT^{T\otimes T_1}(\mathfrak X) \cong \{ \text{$\mathbb
      A$-closed odd forms with values in } V \}.
  \end{gather*}
\end{proposition}

\begin{proof}
  The vector bundle $T\colon \mathfrak B(\mathfrak X) \to
  \mathrm{Vect}$ determines a sheaf $\Gamma_T$ on $\mathfrak
  B(\mathfrak X)$, assigning to an object $f\colon S \to \mathfrak
  B(\mathfrak X)$ the complex vector space of sections of $f^*T$.  The
  bundle $T$ is specified by a coherent family of objects in the
  double cosimplicial groupoid \eqref{eq:19}, representing an object
  in the limit of that diagram, and a global section is specified by a
  coherent family of sections.

  Similarly, the superconnection $\mathbb A$ determines a sheaf
  $\Gamma^*_{\mathbb A}$ on $\mathfrak X$ whose sections over
  $f\colon S \to \mathfrak X$ are the super vector space of forms in
  $\Omega^*(S, f^*V)$ annihilated by $\mathbb A$.  Global sections of
  $\Gamma^*_{\mathbb A}$ are the super vector space
  \begin{equation*}
    \Gamma^*_{\mathbb A}(\mathfrak X) = \lim\left( \Gamma^*_{\mathbb
        A}(X_0) \rightrightarrows \Gamma^*_{\mathbb A}(X_1) \right).
  \end{equation*}
  Now, the data of $(V,\mathbb A)$ is determined, by hypothesis, by
  the same coherent family of objects in \eqref{eq:19} as $T$.
  Suppose we are given a coherent family of sections.  Individually,
  the bottom row of \eqref{eq:19} specifies an element of
  $\Omega^*(X_0,V)$ which is invariant under the $\Isom(\mathbb
  R^{0\vert1}) \cong \mathbb R^{0\vert1} \rtimes \mathbb Z/2$-action;
  this means it is even and closed, i.e., a section of
  $\Gamma^0_{\mathbb A}(X_0)$.  Similarly, the second row by
  itself specifies a section of $\Gamma^0_{\mathbb A}(X_1)$, and
  the coherence conditions involving vertical maps say these two
  things determine a section of $\Gamma^0_{\mathbb A}(\mathfrak
  X)$.  The correspondence between sections of $T$ and
  $\Gamma_{\mathbb A}^0$ is clearly bijective.

  Replacing $T$ with $T \otimes T_1$ in the above argument amounts to
  replacing $V$ with its parity reversal $\Pi V$ (cf.\
  \textcite[proposition 6.3]{MR2763085}).
\end{proof}

Recall that a flat superconnection defines a differential on the space
of forms with values in the corresponding vector bundle.

\begin{proposition}
  \label{prop:13}
  Concordance classes of EFTs are in bijection with cohomology
  classes:
  \begin{equation*}
    0\vert1\EFT^{T\otimes T_n}[\mathfrak X]  \cong H^{\bar n}(\Omega^*(\mathfrak X,V),
    \mathbb A).
  \end{equation*}
\end{proposition}

\begin{proof}
  By naturality of the correspondences in the previous proposition, it
  suffices to show that
  \begin{equation*}
    \Gamma_{\mathbb
      A}^{\bar n}(\mathfrak X)/\mathrm{concordance} \cong H^{\bar n}(\Omega^*(\mathfrak X,V),
    \mathbb A).
  \end{equation*}
  Suppose, first, that the closed forms $\omega_0, \omega_1 \in
  \Omega^*(\mathfrak X, V)$ are cohomologous, i.e., $\omega_1-\omega_0
  = \mathbb A\alpha$.  Then
  \begin{equation*}
    \omega = \omega_0 + \mathbb A(t\alpha) \in \Omega^*(\mathfrak
    X\times \mathbb R, V)
  \end{equation*}
  satisfies $i_j^*\omega = \omega_j$, $j=0,1$.  (Here, we use the same
  notation for an object over $\mathfrak X$ and its pullback via
  $\pr_1\colon \mathfrak X \times \mathbb R \to \mathfrak X$; as
  usual, $t$ is the coordinate on $\mathbb R$.)

  Conversely, suppose we are given a closed form $\omega \in
  \Omega^*(\mathfrak X \times \mathbb R, V)$ with $\omega_j =
  i_j^*\omega$, $j=0,1$.  We need to find a form $\alpha \in
  \Omega^*(\mathfrak X, V)$ such that $\omega_1-\omega_0 = \mathbb A
  \alpha$.  Schematically, it will be
  \begin{math}
    \alpha = -\int_{\mathfrak X \times [0,1]/\mathfrak X} \omega.
  \end{math}
  More precisely, we need to define $\alpha_f \in \Omega^*(S, f^*V)$
  for each $S$-point $f\colon S \to \mathfrak X$.  That will be given
  by the fiberwise integral
  \begin{equation*}
    \alpha_f = - \int_{S \times [0,1]/S} (f\times\mathrm{id})^*\omega,
  \end{equation*}
  which is clearly natural in $S$.  Notice that the vector bundle in
  which $\omega$ takes values comes with a canonical trivialization
  along the $\mathbb R$-direction, so the integral makes sense.

  Now, define operators $\mathbb A_f = f^*\mathbb A \otimes 1$ and
  $d_{\mathbb R} = 1 \otimes d$ on
  \begin{math}
    \Omega^*(S \times \mathbb R, V) \cong \Omega^*(S, V) \otimes
    \Omega^*(\mathbb R).
  \end{math}
  From the derivation property of $\mathbb A$, it follows that $(f
  \times \mathrm{id})^*\mathbb A = \mathbb A_f + d_{\mathbb R}$.
  Then, writing $\omega_f=(f\times\mathrm{id})^*\omega$, we have
  \begin{align*}
    f^*\mathbb A(\alpha_f)
    &= -f^*\mathbb A \int_{S\times[0,1]/S} \omega_f
      = -\int_{S\times[0,1]/S} \mathbb A_f\omega_f
      = \int_{S\times[0,1]/S} d_{\mathbb R} \omega_f\\
    &= f^*\omega_1 - f^*\omega_0.
  \end{align*}
  Thus $\omega_1-\omega_0=\mathbb A\alpha$.
\end{proof}

Applying the theorem to the trivial twist $T$, we get the following.

\begin{corollary}
  \label{cor:0|1-EFT-orbifold}
  For any differentiable stack $\mathfrak X$, there is a natural
  bijection
  \begin{equation*}
    0\vert1\EFT^{T_n}(\mathfrak X) \cong \Omega^{\bar n}_{\mathrm{cl}}(\mathfrak X)
  \end{equation*}
  between $T_n$-twisted $0\vert1$-EFTs over $\mathfrak X$ and closed
  differential forms of parity $\bar n$ on $\mathfrak X$.  If
  $\mathfrak X$ is an orbifold, passing to concordance classes gives
  an isomorphism with $\mathbb Z/2$-graded delocalized cohomology
  \begin{equation*}
    0\vert1\EFT^{T_n}[\Lambda\mathfrak X]
    \cong H^{\bar n}_{\mathrm{deloc}}(\mathfrak X).  
  \end{equation*}
\end{corollary}

\begin{remark}
  Replacing $\mathfrak B(\mathfrak X)$ with the stack of connected
  $0\vert1$-dimensional manifolds over the orbifold $\mathfrak X$,
  \begin{equation*}
    \mathfrak B_{\mathrm{top}}(\mathfrak X) = \underline{\Fun}_{\mathrm{SM}}(\mathbb
    R^{0\vert1}, \mathfrak X)\sslash\mathrm{Diff}(\mathbb R^{0\vert1}),
  \end{equation*}
  we arrive at the notion of \emph{topological} twists and (twisted)
  field theories.  The basic twist $T_1$ lifts in a natural way to
  $\mathfrak B_{\text{top}}(\mathfrak X)$, and in this case $T_n$ in
  fact depends on $n$, and not just on its parity.  In an entirely
  analogous way to \textcite{MR2763085}, one can show that
  \begin{equation*}
    0\vert1\TFT^{T_n}(\mathfrak X) \cong \Omega^n_{\mathrm{cl}}(\mathfrak X),\quad
    0\vert1\TFT^{T_n}[\Lambda\mathfrak X] \cong H^n_{\mathrm{deloc}}(\mathfrak X),
  \end{equation*}
  where the latter identification requires the assumption that
  $\mathfrak X$ is an orbifold.
\end{remark}

\subsection{Twisted de Rham cohomology for orbifolds}
\label{sec:tu-xu-de-rham}

In this section, we review the construction of delocalized twisted de
Rham cohomology for orbifolds due to \textcite{MR2271013}, and show
that it can, in fact, be interpreted as concordance classes of
suitably twisted $0\vert1$-EFTs.  In view of
proposition~\ref{prop:13}, this is not necessarily surprising; the
point here is to give explicit descriptions allowing us to show, in
section~\ref{sec:dimens-reduct-twists}, how the relevant twist arises,
in a natural way, from dimensional reduction.

Let $\mathfrak X$ be an orbifold and $\tilde{\mathfrak X}$ a gerbe
with Dixmier--Douady class $\alpha \in H^3(\mathfrak X,\mathbb Z)$,
both with presentations as in appendix~\ref{sec:twisted-superconn}.
Then the ($\mathbb Z/2$-graded) delocalized twisted cohomology groups
$H_{\mathrm{deloc}}^*(\mathfrak X, \alpha)$ are defined to be the
cohomology of the complex
\begin{equation}
  \label{eq:20}
  (\Omega^*(\Lambda\mathfrak X, L'), \nabla' + \Omega\wedge\cdot).
\end{equation}
Here, $\Omega^*$ stands for the $\mathbb Z/2$-graded de Rham complex,
$\Lambda\mathfrak X$ is the inertia orbifold, $\Omega$ is the
$3$-curvature of $\tilde{\mathfrak X}$ pulled back to
$\Lambda\mathfrak X$, and $(L', \nabla')$ is a line bundle with flat
connection we will describe below.  (This differs from the definition
of \textcite[section 3.3]{MR2271013} in that we perform the usual
trick to convert between $\mathbb Z/2$-graded and $2$-periodic
$\mathbb Z$-graded complexes, and we have chosen a more convenient
constant in front of $\Omega$, which produces an isomorphic chain
complex.  We also remark that changing the gerbe with connective
structure representing the class $\alpha$ produces a noncanonically
isomorphic complex; a specific isomorphism between the complexes
depends on the choice of isomorphism between the gerbes.)

The line bundle $L'$ on the inertia groupoid $\hat X_1
\rightrightarrows \hat X_0$ is as follows:
\begin{enumerate}
\item the underlying line bundle $L'$ is the restriction of $L$ to
  $\hat X_0 \subset X_1$, with the restricted connection $\nabla'$;
\item the isomorphisms $s^*L' \to t^*L'$ over $\hat X_1$ are described
  fiberwise by the composition
  \begin{equation*}
    L'_{(x,g)} = L_g \to L_f \otimes L_g \otimes L_{f^{-1}}
    \to L_{fgf^{-1}} = L'_{(x',g')},
  \end{equation*}
  where $f \in X_1$ induces a morphism $f\colon (x,g) \to (x',g')$ in
  $\hat X_1 \rightrightarrows \hat X_0$, and we used the canonical map
  $\mathbb C \to L_f \otimes L_{f^{-1}}$.
\end{enumerate}
It is an exercise to check that $\nabla'$ is invariant, and flat
(provided our central extension admits a curving, which we always
assume).

Now, the operator $\nabla' + \Omega \wedge \cdot$ on
$\Omega^*(\Lambda\mathfrak X, L')$ is a flat superconnection on $L'
\in \mathrm{Vect}(\Lambda\mathfrak X)$ (since $\nabla'$ is flat,
$d\Omega = 0$ and $\Omega \wedge \Omega = 0$), and therefore gives
rise to a twist $T_\alpha\colon \mathfrak B(\Lambda\mathfrak X) \to
\mathrm{Vect}$.  Combining proposition~\ref{prop:13} with the main
result of \textcite{MR2271013}, we obtain a field-theoretic
interpretation of complexified twisted $K$-theory.  (The compactness
assumption can be dropped by using field theories and de Rham
cohomology with compact support.)

\begin{theorem}
  \label{thm:2}
  For every compact orbifold $\mathfrak X$ and $\alpha \in
  H^3(\mathfrak X,\mathbb Z)$, there are natural bijections
  \begin{gather*}
    0\vert1\EFT^{T_\alpha}[\Lambda\mathfrak X]
    \cong H^{\mathrm{ev}}_{\mathrm{deloc}}(\mathfrak X,\alpha)
    \cong K^\alpha(\mathfrak X) \otimes \mathbb C,\\
    0\vert1\EFT^{T_\alpha \otimes T_1}[\Lambda\mathfrak X]
    \cong H^{\mathrm{odd}}_{\mathrm{deloc}}(\mathfrak X,\alpha)
    \cong K^{1+\alpha}(\mathfrak X) \otimes \mathbb C.
  \end{gather*}
\end{theorem}

\begin{remark}
  All objects indexed by $\alpha$ actually depend, up to noncanonical
  isomorphism, on the choice of a gerbe representative and its
  connective structure.  This abuse of language is standard in the
  literature.
\end{remark}

To finish this section, we rephrase the description of $(L', \nabla')$
in terms of a Deligne $2$-cocycle $(h,A,B)$ on $X_1\rightrightarrows
X_0$ representing $\tilde {\mathfrak X}$ (see
\eqref{eq:9}--\eqref{eq:11} for our notation).  Then $L'$ is
topologically trivial and the connection $\nabla'$ is $d +
A\vert_{\hat X_0}$; flatness is due to the fact that $dA\vert_{\hat
  X_0} = (t^*B - s^*B)\vert_{\hat X_0} = 0$ since $s=t$ on $\hat X_0$.
To describe the isomorphism $s^*L' \to t^*L'$, we use as input $h \in
C^\infty(X_2,\mathbb C^\times)$, and we just need to specify a
$\mathbb C^\times$-valued function $H$ on $\hat X_1$.  Let $v = (g,z)
\in X_1 \times \mathbb C$ and $\tilde f = (f, w) \in X_1\times \mathbb
C$.  Then $\tilde f^{-1}=(f^{-1}, w^{-1} h^{-1}(f,f^{-1}))$, and we
find that $\tilde f v = (fg, zwh(f,g))$ and
\begin{align*}
  \tilde f v \tilde f^{-1}
  & = (fgf^{-1}, zwh(f,g)w^{-1}h(f,f^{-1})h(fg,f^{-1}))\\
  & = (g', zh(f,g)h^{-1}(f,f^{-1})h(fg,f^{-1})).
\end{align*}
Thus, using the cocycle condition \eqref{eq:9} for the triple
$(g',f,f^{-1})$, we get
\begin{equation}
  \label{eq:14}
  \begin{aligned}
    H\left(g\overset f\to g'\right) & = h(f,g)h^{-1}(f,f^{-1})h(fg,f^{-1})\\
    & = \frac{h(f,g)}{h(fgf^{-1},f)}.
  \end{aligned}
\end{equation}

\section{Torsors and bordisms over an orbifold}
\label{sec:bordisms-over-orbifold}

In this section, we provide a manageable model of the bordism category
$1\vert1\EBord(\mathfrak X)$ as a category internal to symmetric
monoidal stacks.  This will also fix notation used in the remainder of
the paper.

\subsection{Basic definitions}
\label{sec:basic-definition}

We start recalling the construction of Euclidean bordism categories
over a manifold $X$, then note that it immediately generalizes to the
case of bordisms over a stack $\mathfrak X$, and finally recast the
result in the language of torsors for a given Lie groupoid
presentation of $\mathfrak X$.

Given integers $d, \delta \geq 0$ (subject to certain conditions) and
a manifold $X$, \textcite{MR2742432} construct a bordism category
$d\vert\delta\EBord(X)$, which we briefly review here.  It is a
category internal to the category of symmetric monoidal stacks; that
is, it is given by symmetric monoidal stacks
$d\vert\delta\EBord(X)_i$, $i=0,1$, called the stack of objects and
the stack of morphisms respectively, together with functors
\begin{equation*}
  \xymatrix{
    d\vert\delta\EBord(X)_1
    \times^{s,t}_{d\vert\delta\EBord(X)_0}
    d\vert\delta\EBord(X)_1 \ar[d]^c
  \\ d\vert\delta\EBord(X)_1 \ar@<-5ex>[d]^s\ar@<5ex>[d]^t
  \\ d\vert\delta\EBord(X)_0 \ar[u]_u,
  }
\end{equation*}
standing for composition, source, target and unit, satisfying the
expected conditions up to prescribed natural transformations
(associator and left and right unitors, similar to the data of a
bicategory).  In the stack of objects $d\vert\delta\EBord(X)_0$, an
object lying over $S$ is given by the following collection of data:
\begin{enumerate}
\item a submersion $Y \to S$ with $d\vert\delta$-dimensional fibers
  and Euclidean structure (in the sense of
  \cite[section~4.2]{MR2742432} or, more succinctly,
  \cite[appendix~B]{arXiv:1703.00314});
\item a map $f\colon Y \to X$;
\item a submersion $Y^c \to S$ with $(d-1\vert\delta)$-dimensional
  fibers, fiberwise embedding $Y^c \to Y$, and a decomposition
  $Y\setminus Y^c = Y^+ \amalg Y^-$.
\end{enumerate}
The $S$-family $Y^c$ is called the \emph{core}.  A morphism in the
stack of objects is given by a germ (around the cores) of $(G, \mathbb
M)$-isometries respecting the maps to $X$.  Thus, $Y^\pm$ should be
thought as germs of collar neighborhoods of the core; they are a
technical device needed, among other things, to define the composition
functor $c$.  In the stack of morphism $d\vert\delta\EBord(X)_1$, an
object lying over $S$ is given by the following collection of data:
\begin{enumerate}
\item a submersion $\Sigma \to S$ with $d\vert\delta$-dimensional
  fibers and Euclidean structure;
\item a map $f\colon \Sigma \to X$;
\item objects $(Y_{\mathrm{in}}, Y_{\mathrm{in}}^c,
  Y_{\mathrm{in}}^\pm)$, $(Y_{\mathrm{out}}, Y_{\mathrm{out}}^c,
  Y_{\mathrm{out}}^\pm)$ of $d\vert\delta\EBord(X)_0$;
\item isometries $Y_{\mathrm{in}} \to \Sigma$ and $Y_{\mathrm{out}}\to
  \Sigma$ respecting the maps to $X$.
\end{enumerate}
The maps of item (4) are ``parametrizations of the boundary'', and are
subject to certain conditions formalizing this idea.  A morphism in
the stack of morphisms is given by (1) isomorphisms in the object
stack between the respective incoming and outgoing boundaries and (2)
an isometry between the $\Sigma$'s (or, more precisely, germs of
isometries around their cores, that is, the region between
$Y^c_{\mathrm{in}}$ and $Y^c_{\mathrm{out}}$), respecting the maps to
$X$ and the parametrizations of the boundaries.  The symmetric
monoidal structures in the stacks of objects and morphisms are given
by fiberwise disjoint union.

Now, turning to aspects more specific to this paper, we note that is
easy to extend the above definition of bordism category to the case
where $X$ is replaced by a ``generalized manifold'', or stack
$\mathfrak X$: an object in $d\vert\delta\EBord(\mathfrak X)_1$ is
given by an object in $d\vert\delta\EBord(\mathrm{pt})_1$ together
with an object of $\mathfrak X_\Sigma$ (which, by the Yoneda lemma,
corresponds to a map $\psi\colon \Sigma \to \mathfrak X$ in the realm
of generalized manifolds) and the corresponding boundary information,
which we will not detail here.  A morphism over $f\colon S' \to S$ in
the stack of bordisms is determined by a fiberwise isometry $F\colon
\Sigma' \to \Sigma$ covering $f$ (and suitably compatible with the
boundary information) together with a morphism $\xi$ between objects
of $\mathfrak X_{\Sigma'}$ as indicated in the diagram below.
\begin{equation}
  \label{eq:5}
  \begin{gathered}
    \xymatrix@R-1.5em{ \Sigma'
      \ar[dd]\ar[rd]^F
      \ar@/^3ex/[rrd]^{\psi'}_{}="a"&&\\
      & \Sigma \ar[dd]\ar[r]^{\psi}  \ar_-\xi@{<=}"a"& \mathfrak X\\
      S' \ar[rd]^f &&\\
      & S &}
  \end{gathered}
\end{equation}

\begin{remark}
  The bordisms-over-stacks point of view is very natural from the
  perspective of geometric structures.  The treatment of rigid
  geometries in \cite{MR2742432}, and in particular Euclidean
  structures, can be interpreted as the definition of a stack of
  atlases.  Then, letting $\mathfrak X$ denote the stack of Euclidean
  atlases, we recover $1\vert1\EBord$ from the plain topological
  bordism category as $1\vert1\Bord(\mathfrak X)$.  We will further
  develop this idea elsewhere.
\end{remark}

Finally, we assume $\mathfrak X$ is a differentiable stack with Lie
groupoid presentation $X = (X_1\rightrightarrows X_0)$, and recall a
convenient way to describe maps $\Sigma \to \mathfrak X$, namely as
$X$-torsors.  (See \textcite[section~2.4]{MR2817778} for a full
account of the theory of torsors.)  An $X$-torsor over $\Sigma$ is
given by
\begin{enumerate}
\item a submersion $\pi\colon U \to \Sigma$,
\item an anchor map $\psi_0\colon U \to X_0$, and
\item an action map $\mu\colon U \times_{X_0} X_1 \to U$.
\end{enumerate}
The conditions required of $\mu$ make it equivalent to the data of a
map $\psi_1$ such that
\begin{equation*}
  \xymatrix{U \times_\Sigma U \ar[r]^-{\psi_1} \ar@<-1ex>[d]\ar@<+1ex>[d] &
    X_1 \ar@<-1ex>[d]\ar@<+1ex>[d]\\
    U \ar[r]^-{\psi_0} & X_0}
\end{equation*}
is an internal functor satisfying the condition that
\begin{equation*}
  (\pr_1, \psi_1)\colon U \times_\Sigma U \rightarrow U \times_{X_0}^{\psi_0,t} X_1
\end{equation*}
is a diffeomorphism; in that case, $\mu$ can be recovered as the
inverse to the above followed by projection onto the second factor.
We will often denote the torsor simply by $\psi$, and write $(\Sigma,
\psi)$ for a bordism in $1\vert1\EBord(\mathfrak X)$.

A morphism of torsors is an equivariant map between the corresponding
$U$'s.  Thus, given a second bordism $(\Sigma', \psi')$ equipped with
an $X$-torsor, which, more specifically, is given by the data
\begin{equation*}
  \Sigma' \to S',\quad
  \pi'\colon U' \to \Sigma',\quad
  \psi_0',\quad
  \mu',
\end{equation*}
a morphism $(F,\lambda)\colon (\Sigma, \psi) \to (\Sigma', \psi')$ in
$1\vert1\EBord(\mathfrak X)_1$ covering $f\colon S \to S'$ is
determined by a fiberwise isometry $F\colon \Sigma \to \Sigma'$
covering $f$ and compatible with the boundary data, together with an
equivariant map $\lambda\colon U \to U'$ covering $F$ and compatible
with the anchor maps: $\psi_0' \circ \lambda = \psi_0$ (this, again,
is taken up to a suitable germ equivalence relation).  When
$\pi'=\pi$, the datum of $\lambda$ is equivalent to an internal
natural transformation between the internal functors determined by
$\psi$, $\psi'$; namely we set $\Lambda\colon U \to X_1$ to be the
composition
\begin{equation*}
  U \xrightarrow{(\mathrm{id}, \lambda)} U \times_S U
  \xrightarrow[\mu'\text{-action}]{\cong}U\times_{X_0}^{\psi_0',t}X_1
  \xrightarrow{\pr_2} X_1.
\end{equation*}

The stack of objects $1\vert1\EBord(\mathfrak X)_0$ has an analogous
description. To fix our notation, which closely follows the previous
discussion, an object here is given by (1) an $S$-family $Y \to S$ of
$1\vert1$-manifolds, (2) a codimension $1$ family of submanifolds $Y^c
\subset Y$, called the core, (3) fiberwise Euclidean structures on the
pair $(Y, Y^c)$, (4) a decomposition $Y \setminus Y^c = Y^+ \amalg
Y^-$, and (5) an $X$-torsor
\begin{equation*}
  \pi\colon U \to Y, \quad
  \psi_0\colon U \to X_0, \quad
  \mu\colon U \times_{X_0} X_1 \to U
\end{equation*}
on $Y$.  We typically write $(Y, \psi)$ for such an object.  A
morphism $(Y, \psi) \to (Y', \psi')$ is given by the germ (around
$Y^c$) of an isometry $F\colon Y \to Y'$ together with the germ
(around $\pi^{-1}(Y^c)$) of an equivariant map $\lambda_0\colon U \to
U'$.

\subsection{Skeletons}
\label{sec:skeletons}

To get an intuitive understanding of bordism categories over a stack
$\mathfrak X$, avoiding torsors, we can think as follows.  First, fix
a Lie groupoid $X_1 \rightrightarrows X_0$ presentation of $\mathfrak
X$.  Then some bordisms and isometries between them can be represented
by pictures like the following.
\begin{equation}
  \label{eq:7}
  \begin{gathered}
    \xymatrix{\Sigma \ar[r]^\psi\ar[d]& X_0\\S}\qquad \xymatrix{
      \Sigma'\ar[d]\ar[r]^F & \Sigma \ar[d]\ar[r]^{\xi}&  X_1\\
      S' \ar[r]& S&}
  \end{gathered}
\end{equation}
In fact, the $\psi$ and $\xi$ above relate to their counterparts in
\eqref{eq:5} by postcomposition with the atlas $X_0 \to \mathfrak X$
respectively whiskering with the natural transformation between the
two maps $X_1\rightrightarrows \mathfrak X$.  Not every bordism is of
this form, but in a fully extended framework it is intuitively clear
that every bordism can be expressed as a composition of those; it is
also not hard to conceive relations between those basic building
blocks.  The notion of skeletons, which we introduce now, is
essentially a way of dealing systematically with these generators and
relations in our case of interest.

From now on, we assume that $X$ is an orbifold groupoid, so that the
submersions $\pi$ underlying all $X$-torsors are étale.  We then
define a \emph{skeleton} of a fiberwise connected $S$-family $(Y,\psi)
\in 1\vert1\EBord(\mathfrak X)_0$ to be given by a map
\begin{math}
  \iota\colon S \times \mathbb R^{0\vert1} \to U
\end{math}
such that the composition $\pi \circ \iota$ gives a Euclidean
parametrization of the core $Y_c$.  In general, a skeleton is given by
a skeleton for each connected component.

A skeleton of a fiberwise connected $S$-family $(\Sigma, \psi) \in
1\vert1\EBord(\mathfrak X)_1$ is given by skeletons for the incoming
and outgoing boundary components, together with the following:
\begin{enumerate}
\item A collection $I_i$, $0\leq i\leq n$, of $S$-families of
  superintervals, as defined in appendix~\ref{sec:fund-theor-calc}.
  We denote the inclusion of the outgoing and incoming boundary
  components by
  \begin{equation*}
    S \times \mathbb R^{0\vert1} \overset{\iota_i^{\mathrm{out}}}\hookrightarrow S \times
    \mathbb R^{1\vert1} \overset{\iota_i^{\mathrm{in}}}\hookleftarrow S \times \mathbb R^{0\vert1}.
  \end{equation*}
  It is not required that these intervals have strictly positive
  length.
\item For each $i$, an embedding $I_i \hookrightarrow U$, by which we
  mean a Euclidean map from a neighborhood of the ``core''
  $[\iota_i^{\mathrm{out}}, \iota_i^{\mathrm{in}}] \subset S \times
  \mathbb R^{1\vert1}$ of $I_i$ into $U$.
\end{enumerate}
Here and in what follows, we write
\begin{gather*}
  b_i\colon S \times\mathbb R^{0\vert1} \xrightarrow{\iota_i^{\mathrm{out}}}
  S\times\mathbb R^{1\vert1} \hookrightarrow U,\\
  a_i\colon S \times\mathbb R^{0\vert1} \xrightarrow{\iota_i^{\mathrm{in}}}
  S\times\mathbb R^{1\vert1} \hookrightarrow U.
\end{gather*}
We require the following conditions of the above data:
\begin{enumerate}
\item for each $1\leq i< n$, the maps $\pi \circ b_i$ and $\pi \circ
  a_{i+1} \colon S \times \mathbb R^{0\vert1} \to \Sigma$ coincide;
\item if $\Sigma$ is a family of supercircles, then we also have
  $\pi\circ b_n = \pi \circ a_0$;
\item if $\Sigma$ has boundary, then the maps $a_0$ and $b_n$,
  together with the parametrization of the boundary, induce skeletons
  on each boundary component; we require that those agree with the
  initially given skeletons.
\end{enumerate}

These conditions mean, intuitively, that the superintervals $I_i
\hookrightarrow U \xrightarrow\pi \Sigma$ prescribe an expression of
$\Sigma$ as a composition of shorter pieces (right elbows and length
zero left elbows) in $1\vert1\EBord$, together with expressions of
each of these pieces in the form \eqref{eq:7}.

We will use the shorthand notation $I = \{ I_i \}$ to refer to the
skeleton, and $(\Sigma, \psi, I)$ to refer to a bordism with a choice
of skeleton.

Notice that $b_{i-1}$ and $a_i$ are isomorphic in the groupoid of $(S
\times \mathbb R^{0\vert1})$-points of $U \times_\Sigma U
\rightrightarrows U$, since their images in $\Sigma$ agree; we denote
by $j_i\colon S \times \mathbb R^{0\vert1} \to U \times_\Sigma U$ the
unique morphism $b_{i-1} \to a_i$, that is, the unique map such that
\begin{equation}
  \label{eq:13}
  \pr_1 \circ j_i = b_{i-1}, \quad \pr_2 \circ j_i = a_i.
\end{equation}

We denote by $1\vert1\EBord(\mathfrak X)^{\mathrm{skel}}_i$, $i=0,1$,
the variants of $1\vert1\EBord(\mathfrak X)_i$ where each bordism and
boundary component comes with a choice of skeleton; morphisms in these
stacks are just morphisms in the old variants, after forgetting the
skeleton.  There is a canonical choice of skeleton on the composition
of bordisms with skeleton.  With this observation, we obtain an
internal category $1\vert1\EBord(\mathfrak X)^{\mathrm{skel}}$.

\begin{proposition}
  \label{prop:4}
  The forgetful map $1\vert1\EBord(\mathfrak X)^{\mathrm{skel}} \to
  1\vert1\EBord(\mathfrak X)$ is a levelwise equivalence.
\end{proposition}

This is clear, since all spaces of skeletons are contractible.  It is
also clear that $1\vert1\EBord(\mathfrak X)$ does not depend on the
choice of a Lie groupoid presentation for $\mathfrak X$, since it only
makes reference to torsors over it.  On the other hand, the definition
of $1\vert1\EBord(\mathfrak X)^{\mathrm{skel}}$ does make explicit
reference to the groupoid $X_1 \rightrightarrows X_0$, so the notation
is slightly abusive.  This is harmless, as shown by the previous
proposition.

\begin{remark}
  Evidently, we can form a pullback of $(\Sigma, \psi, I) \in
  1\vert1\EBord(\mathfrak X)^{\mathrm{skel}}_1$ via a map $f\colon S'
  \to S$ by simply choosing a cartesian morphism $\lambda\colon
  (\Sigma', \psi') \to (\Sigma, \psi)$ covering $f$ in
  $1\vert1\EBord(\mathfrak X)_1$ and any skeleton $I'$ for $(\Sigma',
  \psi')$.  However, we note that there is a canonical choice to be
  made: we ask that $I, I'$ have the same indexing set and
  \begin{equation*}
    \xymatrix{ I'_i \ar[r]\ar[d] & I_i \ar[d]\\ U' \ar[r]^\lambda& U}
  \end{equation*}
  is cartesian for all $i$.  We denote that skeleton by $\lambda^*I$,
  the endpoints $a'_i$ by $\lambda^*a_i$, etc.
\end{remark}

We call the collection of maps $I_i \to U \overset\pi\to \Sigma$ the
associated triangulation of the skeleton $I$.  Triangulations such
that the diagrams
\begin{equation*}
  \xymatrix{ I'_i \ar[r]\ar[d] & I_i \ar[d]\\ \Sigma' \ar[r]^F& \Sigma}
\end{equation*}
are cartesian will be called \emph{compatible}.

Finally, suppose we have $(\Sigma, \psi) \in 1\vert1\EBord(\mathfrak
X)_1$ and
\begin{equation*}
  I = \{ I_i \}_{i \in \mathcal I}, \quad I' = \{ I_i' \}_{i \in \mathcal I'}
\end{equation*}
two skeletons.  Then we say that $I'$ is refinement of $I$ if there is
a surjective map $r\colon\mathcal I' \to \mathcal I$, such that, for
each $i \in \mathcal I$, $r^{-1}(i)$ indexes a collection $I_{i_1}',
\dots, I_{i_n}' \subset U$ where $b_{i_k}' = a_{i_{k+1}}'$ for each
$1\leq k < n$ and $a_{i_1} = a_i'$, $b_{i_n} = b_i'$; in words, the
$I_{i_k}'$, $1\leq k\leq n$, are adjacent subintervals whose
concatenation is precisely $I_i$.  We denote by $R_{I'}^I\colon (\psi,
I') \to (\psi, I)$ the morphism having the identity as its underlying
torsor map.

\subsection{The globular subcategory; the superpath stack}
\label{sec:globular-subcategory}

Denote by
\begin{equation*}
  1\vert1\EBord(\mathfrak X)_0^{\mathrm{glob}} \subset
  1\vert1\EBord(\mathfrak X)_0^{\mathrm{skel}}
\end{equation*}
the sub-prestack with the same objects but containing only those
morphisms $(F,\lambda)\colon (Y, \psi) \to (Y', \psi')$ such that the
diagram
\begin{equation}
  \label{eq:17}
  \xymatrix{
    S \times \mathbb R^{0\vert1} \ar[r]^{f\times \id} \ar[d]^\iota&
    S' \times \mathbb R^{0\vert1} \ar[d]^{\iota'}\\
    U \ar[r]^\lambda & U'}
\end{equation}
commutes, where $f\colon S \to S'$ is the map $F$ lies over.  Denote
by $1\vert1\EBord(\mathfrak X)_1^{\mathrm{glob}} \subset
1\vert1\EBord(\mathfrak X)_1^{\mathrm{skel}}$ the sub-prestack
containing only those morphisms that map into $1\vert1\EBord(\mathfrak
X)_0^{\mathrm{glob}}$ via the source and target functors.  These two
objects fit together into a “globular” internal category
$1\vert1\EBord(\mathfrak X)^{\mathrm{glob}}$, which can be thought of
as a smooth bicategory.

For each test manifold $S$, we obtain from each of the above variants
a category
\begin{equation*}
  1\vert1\EBord(\mathfrak X)_S^{\mathrm{glob}}, \qquad
  1\vert1\EBord(\mathfrak X)_S^{\mathrm{skel}}
\end{equation*}
internal to symmetric monoidal groupoids.  Those internal categories
are fibrant in the sense of \textcite{arXiv:1004.0993}, and they
clearly determine the same symmetric monoidal bicategory.  Thus, the
inclusion $1\vert1\EBord(\mathfrak X)^{\mathrm{glob}} \to
1\vert1\EBord(\mathfrak X)^{\mathrm{skel}}$ ought to be considered as
an equivalence of internal categories.  Since we do not know of a
comprehensive theory of internal categories to quote from, we will
leave this as an informal statement.

Our construction of twists and twisted field theories below will be
based on the globular variant.  This provides some slight
simplifications and allows us to focus on the more conceptual side of
the discussion.  More specifically, in order to extend our
construction of the twist functor $T$ in section~\ref{sec:twists} from
$1\vert1\EBord(\mathfrak X)^{\mathrm{glob}}$ to
$1\vert1\EBord(\mathfrak X)^{\mathrm{skel}}$, it is necessary (and
sufficient) to choose a stable trivialization of the gerbe
$\tilde{\mathfrak X}$ (passing, if needed, to a finer Lie groupoid
presentation of $\mathfrak X$), as the reader unsatisfied with the
argument of the previous paragraph should be able to verify.

\subsection{Some examples}
\label{sec:generators-relations}

We give here an essentially complete description of $1\vert1\EBord =
1\vert1\EBord(\mathrm{pt})$.  Examples of objects and morphisms in
$1\vert1\EBord(\mathfrak X)$ can then be constructed by pulling back
to a general base space $S$ and choosing a torsor representing a map
to $\mathfrak X$.

The object stack $1\vert1\EBord_0$ contains an object $\mathrm{sp}$
given by the manifold $Y = \mathbb R^{1\vert1}$ with core $Y^c =
\mathbb R^{0\vert1}$ and neighborhoods $Y^+ = \mathbb
R^{1\vert1}_{>0}$, $Y^- = \mathbb R^{1\vert1}_{<0}$.  There is a
similar (but nonisomorphic) object $\overline{\mathrm{sp}}$ with $Y^+$
and $Y^-$ interchanged.  Both of them have $\mathbb Z/2$ as
automorphism group, generated by the flip $\mathrm{fl}\colon \mathbb
R^{1\vert1} \to \mathbb R^{1\vert1}$ which acts by $-1$ on odd
functions.

Objects in $1\vert1\EBord_1$ are given by fiberwise disjoint unions of
one of four kinds of basic bordisms: superintervals, left and right
elbows, and supercircles.  The basic building blocks are as follows.
\begin{enumerate}
\item The left elbow of length $0$, $L_0 \in
  1\vert1\EBord(\overline{\mathrm{sp}} \amalg \mathrm{sp},
  \emptyset)$, has $\mathbb R^{1\vert1}$ as underlying manifold.  The
  boundary parametrization
  \begin{math}
    L_0 \leftarrow \overline{\mathrm{sp}} \amalg \mathrm{sp}
  \end{math}, in terms of the underlying manifolds, is the map
  $\mathrm{id} \amalg \mathrm{id}$.
\item The right elbows $R_r \in 1\vert1\EBord(\emptyset, \mathrm{sp}
  \amalg \overline{\mathrm{sp}})$ form a $\mathbb
  R^{1\vert1}_{>0}$-family, pa\-ram\-e\-tri\-zing the ``super length''
  $r$.  The underlying family of Euclidean manifolds is $\mathbb
  R^{1\vert1}_{>0} \times \mathbb R^{1\vert1} \to \mathbb
  R^{1\vert1}_{>0}$, and the parametrization of the boundary is
  \begin{equation*}
    \mathrm{id} \amalg T_r\colon
    (\mathbb R^{1\vert1}_{>0} \times \mathbb R^{1\vert1}) \amalg (\mathbb
    R^{1\vert1}_{>0} \times \mathbb R^{1\vert1}) \to \mathbb
    R^{1\vert1}_{>0} \times \mathbb R^{1\vert1},
  \end{equation*}
  where $T_r\colon (r, s) \mapsto (r, r \cdot s)$ is the translation
  on affine Euclidean space $\mathbb R^{1\vert1}$ by the amount
  specified by the $\mathbb R^{1\vert1}_{>0}$ parameter.
\item Any isomorphism $F\colon Y \to Y'$ in $1\vert1\EBord_0$ leads to
  a ``thin'' bordism $F \in 1\vert1\EBord_1$, having $Y'$ as
  underlying manifold, $F\colon Y \to Y'$ as incoming parametrization,
  and $\mathrm{id}_{Y'}$ as outgoing parametrization.
\end{enumerate}
Some isomorphisms in $1\vert1\EBord_1$ are listed below.  They
restrict to the identity on the boundaries.  Isomorphisms (1) and (4)
are given by the obvious identification of the underlying manifold of
each bordism, and (2) and (3) by a flip.
\begin{enumerate}
\item $\mathrm{fl}^2 \cong \mathrm{id}_{\mathrm{sp}}$
\item $L_0 \circ (\mathrm{fl} \amalg \mathrm{fl}) \cong L_0$
\item $(\mathrm{fl} \amalg \mathrm{fl}) \circ R_r \cong
  R_{\mathrm{fl}(r)}$
\item $R_{r_1\cdot r_2}
  \cong R_{r_1}
    \circ (\mathrm{id}_{\mathrm{sp}}
           \amalg L_0
           \amalg \mathrm{id}_{\mathrm{\overline{\mathrm{sp}}}})
    \circ R_{r_2}$.
\end{enumerate}
The last of them is an isomorphisms of $(\mathbb R^{1\vert1}_{>0}
\times \mathbb R^{1\vert1}_{>0})$-families, and $r_1, r_2$ indicate
the coordinate function of each factor; more formally, the $R$'s
indicate pullbacks of $R_r$ along the multiplication respectively
projection maps $\mathbb R^{1\vert1}_{>0} \times \mathbb
R^{1\vert1}_{>0} \to \mathbb R^{1\vert1}_{>0}$.  Similar bordisms
$\bar L$, $\bar R$, etc., are obtained by reversing the roles of
$\mathrm{sp}$ and $\overline{\mathrm{sp}}$.  They satisfy analogous
relations to the above, and there are also isomorphisms
\begin{equation*}
  L \circ \sigma \cong \bar L,\qquad \sigma \circ R \cong \bar R.
\end{equation*}

A family $I_r \in 1\vert1\EBord(\mathrm{sp}, \mathrm{sp})$ of
intervals of length $r$ is obtained by composing $L_0$ and $R_r$ along
the common $\overline{\mathrm{sp}}$ boundary:
\begin{equation*}
  I_r = (\mathrm{id}_{\mathrm{sp}} \amalg L_0)
   \circ (R_r \amalg \mathrm{id}_{\mathrm{sp}}).
\end{equation*}
Similarly, a family $K_r \in 1\vert1\EBord(\emptyset, \emptyset)$ of
supercircles of length $r$ is obtained from elbows and the braiding
isomorphism in the way indicated in figure~\ref{fig:circle}.  There
are stacks $\mathfrak P$ and $\mathfrak K$ of superintervals
respectively supercircles.  More generally, we write
\begin{equation*}
  \mathfrak K(\mathfrak X), 
  \mathfrak P(\mathfrak X)
   \hookrightarrow 1\vert1\EBord(\mathfrak X)^{\mathrm{glob}}_1
\end{equation*}
for the stacks of supercircles respectively superintervals over
$\mathfrak X$.

\begin{figure}
  \centering
    \begin{tikzpicture}[baseline={(0,0.5)}]
    \draw[l,L] (0,1) to["$L_0$"left] (0,0);
    \draw[S] ([r] 0,0) to["$\sigma$"right] ([l] 1,1);
    \draw[S, preaction={draw,ultra thick,white}] ([r] 0,1) to ([l] 1,0);
    \draw[r,R] (1,0)
      to[looseness=3, "$R_r$"right] (1,1);
    \draw (0,0) pic {dotm}
          (0,1) pic {dotp} node[above]{$\mathrm{sp}$}
          (1,0) pic {dotp} 
          (1,1) pic {dotm} node[above]{$\overline{\mathrm{sp}}$};
   \end{tikzpicture}
   \caption{A supercircle of length $r$.}
  \label{fig:circle}
\end{figure}
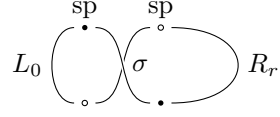

\begin{remark}
  \Textcite[theorem 6.42]{MR2648897} provide generators and relations
  for a variant of $1\vert1\EBord$ (in their case, unoriented and
  satisfying a positivity condition, but, more importantly, not
  extended up to include isometries of bordisms).  Since our goal is
  to give some examples of field theories, and not a classification,
  we will be satisfied with a somewhat informal approach to
  constructing functors between internal categories.  In fact, we will
  explain our constructions in detail on superintervals, and we will
  be less explicit about extending fibered functors on $\mathfrak
  P(\mathfrak X)$ to full-blown internal functors.  Such details are
  usually easy to guess, and compatibility with relations (1)--(4)
  above will be easy to verify.
\end{remark}

\section{Twists for $1|1$-EFTs from gerbes with connection}
\label{sec:twists}

Let $\mathfrak X$ be an orbifold and $\tilde{\mathfrak X} \to
\mathfrak X$ be a gerbe with connective structure, presented by a Lie
groupoid $X = (X_1 \rightrightarrows X_0)$ respectively a central
extension having $L \to X_1$ as the underlying line bundle with
connection, as in appendix~\ref{sec:twisted-superconn}.  The goal of
this section is to associate to $\tilde{\mathfrak X}$ a Euclidean
$1\vert1$-twist
\begin{equation*}
T = T_{\tilde{\mathfrak X}} \in 1\vert1\ETw(\mathfrak X) =
\Fun^\otimes_{\mathrm{SM}}(1\vert1\EBord(\mathfrak X)^{\mathrm{glob}}, \mathrm{Alg}).
\end{equation*}
This construction is drastically simplified by the fact that it takes
values in the subcategory $B\mathrm{Line} \hookrightarrow
\mathrm{Alg}$, where $\mathrm{Line}$ denotes the symmetric monoidal
stack of complex line bundles and $B\mathrm{Line}$ the internal
category having trivial stack of objects and $\mathrm{Line}$ as stack
of morphisms.  In other words, the only relevant algebra in this
construction is the monoidal unit $\mathbb C \in \mathrm{Alg}$, and
the only relevant modules are the invertible ones.

\subsection{Construction of the twist functor}
\label{sec:gerbes-1vert1-twists}

At the level of object stacks, there is no work to do.  We just need
to describe a map $1\vert1\EBord(\mathfrak X)_1^{\mathrm{glob}} \to
\mathrm{Line}$ of symmetric monoidal stacks, which by abuse of
notation we still call $T$, together with natural isomorphisms $\mu$
and $\epsilon$, the compositor and unitor (cf.\ \cite[defintion
2.18]{MR2742432}).  We start discussing the underlying fibered functor
$T$.

Fix a fiberwise connected $S$-family $(\Sigma, \psi, I)$ of bordisms
with skeleton.  Our goal is to describe a line bundle $T(\Sigma, \psi,
I)$ over $S$.  Recall \eqref{eq:13} that the skeleton $I$ determines a
collection of maps $j_i\colon S \times \mathbb R^{0\vert1} \to U
\times_\Sigma U$.  We set
\begin{equation*}
  T(\Sigma, \psi, I) = \bigotimes\nolimits_{1\leq i\leq n} L_{j_i}.
\end{equation*}
Here and in what follows, we write, for any map $f\colon S \times
\mathbb R^{0\vert1} \to U \times_\Sigma U$,
\begin{equation*}
  L_f = (\psi_1\circ f)^*L\vert_{S \times 0}.
\end{equation*}

As to morphisms, we initially consider two cases.
\begin{enumerate}
\item $\lambda\colon (\Sigma, \psi, I) \to (\Sigma', \psi', I')$ is a
  refinement of skeletons, i.e., the underlying torsors are equal, the
  torsor map $\lambda$ is the identity, and $I$ is a refinement of
  $I'$.
\item $\lambda\colon (\Sigma, \psi, I) \to (\Sigma', \psi', I')$
  covers a map $f\colon S \to S'$ and $I$ and $\lambda^*I'$ are
  compatible skeletons. This means that the endpoints of the
  superintervals $I_i, \lambda^*I_i' \subset U$ are (uniquely)
  isomorphic in the groupoid $U \times_K U \rightrightarrows U$, and
  we denote by $\tilde a_i, \tilde b_i \in (U\times_KU)_{S \times
    \mathbb R^{0\vert1}}$ the corresponding morphisms $a_i \to
  \lambda^*a_i'$, $b_i \to \lambda^*b_i'$.  Note that $\tilde a_i$,
  $\tilde b_i$ are the endpoints of the superintervals
  \begin{equation*}
    J_i = I_i \times_K \lambda^*I_i' \subset U \times_KU.
  \end{equation*}
\end{enumerate}
In situation (1), the line bundles $T(\Sigma, \psi, I)$ and $T(\Sigma,
\psi, I')$ only differ by the addition of canonically trivial tensor
factors, and $T(\lambda)$ is the natural identification.  Situation
(2) is more interesting.  We denote by $\mathrm{SP}_i\colon L_{\tilde
  a_i} \to L_{\tilde b_i}$ the super parallel transport of $\psi_1^*L$
along $J_i$, and by
\begin{equation*}
  h_i\colon L_{\tilde a_i} \otimes L_{j_i}
  \to L_{\lambda^*j_i'}  \otimes L_{\tilde b_{i-1}} 
\end{equation*}
the gerbe multiplication map, or any other map obtained by adjunction.
Then we finally consider the composition
\begin{multline}
  \label{eq:2}
  L_{\tilde a_0}^\vee \xrightarrow{\mathrm{SP}^\vee_0} L_{\tilde b_0}^\vee
  \xrightarrow{h_1} L_{\lambda^*j_1'} \otimes L_{\tilde a_1}^\vee \otimes
  L_{j_1}^\vee \xrightarrow{\mathrm{SP}^\vee_1} L_{\lambda^*j_1'} \otimes L_{\tilde b_1}^\vee \otimes
  L_{j_1}^\vee \xrightarrow{h_2}\\
  \xrightarrow{h_2} L_{\lambda^*j_1'} \otimes L_{\lambda^*j_2'} \otimes L_{\tilde a_2}^\vee \otimes
  L_{j_2}^\vee \otimes
  L_{j_1}^\vee \xrightarrow{\mathrm{SP}^\vee_2} \dots \to \\
  \xrightarrow{\mathrm{SP}^\vee_n} \left(
    \bigotimes\nolimits_{1\leq i\leq n} L_{\lambda^*j_i'} \right) \otimes
  L_{\tilde b_n}^\vee \otimes \left( \bigotimes\nolimits_{1\leq i\leq n} L_{j_i}^\vee \right).
\end{multline}
Since we are working with globular bordisms, $\tilde a_0$ and $\tilde
b_n$ are identities, so $L_{\tilde a_0}$, $L_{\tilde b_n}$ are
trivial and we let
\begin{equation}
  \label{eq:1}
  T(\lambda)\colon T(\Sigma, \psi, I) = \bigotimes_{1\leq i\leq n} L_{j_i} 
  \to \bigotimes_{1\leq i\leq n} L_{\lambda^*j_i'} = \lambda^*T(\Sigma', \psi', I')
\end{equation}
be adjoint to the above.

\begin{proposition}
  The prescriptions above uniquely determine a symmetric monoidal
  fibered functor $T\colon 1\vert1\EBord(\mathfrak
  X)_1^{\mathrm{glob}} \to \mathrm{Vect}$.
\end{proposition}

\begin{proof}
  Initially, we will assume we can pick compatible refinements for
  (families of) triangulations of bordisms whenever needed, and
  explain at the end of the proof how to deal with the fact that such
  refinements do not always exist.
  
  Fix a morphism $\lambda\colon (\Sigma, \psi, I) \to (\Sigma', \psi',
  I')$ in $1\vert1\EBord(\mathfrak X)^{\mathrm{glob}}$ and compatible
  refinements for the triangulations of $\Sigma$ and $\Sigma'$.  This
  yields refinements $\bar I$, $\bar I'$ of $I$ respectively $I'$.  We
  can then express $\lambda$ as the composition
  \begin{equation*}
    \xymatrix{(\psi, I) \ar@{<-}[r]^{R_{I}^{\bar I}} \ar@{|->}[d] & (\psi,
      \bar I) \ar[r]^{\bar \lambda}\ar@{|->}[d] & (\psi',
      \bar I')\ar@{|->}[d] & (\psi', I') \ar@{<-}[l]_{R_{I'}^{\bar I'}} \ar@{|->}[d]\\
      S \ar@{=}[r] & S \ar[r]^f& S' & S' \ar@{=}[l]}
  \end{equation*}
  where $\bar\lambda$ is the morphism in $1\vert1\EBord(\mathfrak
  X)_1^{\mathrm{glob}}$ corresponding to the same torsor map as
  $\lambda$, but relating objects with different skeletons.  This
  fixes $T(\lambda)$.

  Of course, we need to check that taking this as a definition for
  $L(\lambda)$ is consistent, that is, independent on the choice of
  $\bar I$ and $\bar I'$.  Verifying this in the case when all
  triangulations involved admit compatible refinements boils down to
  checking that if the original triangulations of $\Sigma$ and
  $\Sigma'$ were already compatible, applying formula \eqref{eq:1}
  would give
  \begin{equation*}
    T(\lambda) \circ T(R_{I}^{\bar I}) =  T(R_{I'}^{\bar I'}) \circ T(\bar\lambda).
  \end{equation*}
  But this is easy to see.
  When calculating $T(\bar\lambda)$, each appearance of $\mathrm{SP}_i$ in
  \eqref{eq:2} gets replaced by a composition
  \begin{equation*}
    \mathrm{SP}^\vee_{i_k} \circ h_{i_k} \circ \dots \circ \mathrm{SP}^\vee_{i_1}
    \circ h_{i_1} \circ  \mathrm{SP}^\vee_{i_0},
  \end{equation*}
  where $\mathrm{SP}_{i_j}$, $0\leq j\leq k$, denotes parallel
  transport along a subdivision $J_{i_j} \subset J_i$ and the
  $h_{i_j}$ are tautological identifications involving $L_{i_j} =
  \mathbb C$.  Thus our claim follows from compatibility of super
  parallel transport with gluing of superintervals.

  Next, we verify that $L$ respects compositions of isometries, at
  least when compatible refinements can be chosen.  So let us fix
  composable morphisms as in the diagram below.
  \begin{equation*}
    \xymatrix{
      (\Sigma, \psi, I) \ar[r]^\lambda \ar@{|->}[d] \ar@/^{4.8ex}/[rr]_{\lambda''}
      & (\Sigma', \psi', I') \ar[r]^{\lambda'}\ar@{|->}[d]
      & (\Sigma'', \psi'', I'')\ar@{|->}[d]
      \\
      S \ar[r]^f\ar@/_{4.2ex}/[rr]^{f''}
      & S' \ar[r]^{f'}
      & S''}
  \end{equation*}
  We can assume the all skeletons are compatible.  Using \eqref{eq:2}
  and the structure maps of the gerbe, we see that (up to braiding),
  \begin{equation*}
    T(\lambda) \otimes f^*T(\lambda') = T(\lambda'') \otimes \Id_{f^*T(\Sigma')}.
  \end{equation*}
  It follows that
  \begin{equation*}
    T(\lambda'') = T(\lambda') \circ T(\lambda).
  \end{equation*}

  Next, suppose we have a morphism $\lambda\colon (\psi, I) \to (\psi',
  I')$ where compatible refinements of the underlying triangulations
  fail to exist.  Since every morphism in $1\vert1\EBord(\mathfrak
  X)_1^{\mathrm{skel}}$ can be expressed as the composition of a
  morphism covering the identity and a morphism involving pullback
  skeletons, it suffices to consider the case when $\lambda$ covers
  $\mathrm{id}\colon S\to S$.  Write $c_i$ for the common value of
  $\pi \circ b_{i-1} = \pi\circ a_i\colon S \times \mathbb R^{0\vert1}
  \to \Sigma$, and define $c_i'$ similarly.  Then the nonexistence of
  a common refinement means that there is a pair $c_i$,
  $\lambda^*c_i'\colon S \times \mathbb R^{0\vert1} \to \Sigma$
  ``crossing over'' one another; more precisely, in a Euclidean local
  chart, neither $(c_i)_{\mathrm{red}}\leq
  (\lambda^*c_i')_{\mathrm{red}}$ nor the opposite holds.  It suffices
  to define $T(\lambda)$ in a small neighborhood $S_p$ of each point
  $p \in S$ where that happens; assuming, for simplicity, that the
  triangulations $\{ c_j \}$, $\{ \lambda^*c_j' \}$ are identical
  except for the problematic index $i$, it suffices to analyze the
  situation in a small neighborhood in $\Sigma$ of the point $x =
  c_i(p)=\lambda^*c_i'(p)$.  Then we can choose $d^1\colon S_p
  \times\mathbb R^{0\vert1}\to \Sigma|_{S_p}$ sufficiently close to
  $c_i$ satisfying either $d(p) < c_i(p), \lambda^*c_i'(p)$ or the
  opposite inequality.  Denote by $I^1$, $I^1{}'$ the modifications of
  $I$, $I'$ obtained by replacing $c_i$, $\lambda^*c_i'$ with $d^1$.
  Then of course $I$ and $I^1$ admit a common refinement, and so do
  $I'$ and $I^1{}'$; moreover, $I^1$ and $I^1{}'$ are based on the
  same triangulation of $\Sigma$.  We have a commutative square
  \begin{equation}
    \label{eq:15}
    \begin{gathered}
      \xymatrix{(\Sigma, \psi,I^1) \ar[r]^{\lambda^1}\ar[d] & (\Sigma', \psi',
        I^1{}')
        \ar[d]\\
        (\Sigma, \psi, I) \ar[r]^{\lambda} & (\Sigma', \psi', I'),}
    \end{gathered}
  \end{equation}
  and this stipulates the value of $T$ on $\lambda\colon(\Sigma, \psi,
  I) \to (\Sigma', \psi', I')$; here, the unlabeled arrows refer to
  morphisms whose underlying torsor maps are the identity.  We need to
  see why this is independent on the choice of $d^1$.  Suppose we
  repeat the above procedure using a different choice $d^2\colon S_p
  \times \mathbb R^{0\vert1}\to \Sigma\vert_{S_p}$; to compare them,
  we can use a third $d^3\colon S_p \times \mathbb R^{0\vert1}\to
  \Sigma\vert_{S_p}$ (restricting, perhaps, to a smaller neighborhood
  $S_p$) which stays away from both $d^1$ and $d^2$.  Thus, we can
  assume without loss of generality that $d^1$ and $d^2$ stay away
  from one another.  We have a commutative diagram
  \begin{equation*}
    \xymatrix@R=1ex{
      &(\Sigma, \psi,I^1)\ar[dd] \ar[ld]\ar[r]^{\lambda^1}&(\Sigma', \psi',I^1{}')
      \ar[rd] \ar[dd]\\
      (\Sigma, \psi,I)&&&(\Sigma', \psi',I')\\
      &(\Sigma, \psi,I^2) \ar[lu]\ar[r]^{\lambda^2}&(\Sigma', \psi',I^2{}') \ar[ru]
    }
  \end{equation*}
  where the skeletons appearing in each triangle and the in middle
  square admit common refinements, and it follows that $T(\lambda)$,
  as prescribed by \eqref{eq:15}, is independent on the choice of
  $d^1$.  Similarly, we can reduce the verification that $T$ respects
  composition of morphisms to the case where all triangulations
  involved admit compatible refinements.

  So far, we have defined $T$ on fiberwise connected families of
  bordisms.  Since the symmetric monoidal structure of
  $1\vert1\EBord(\mathfrak X)_1^{\mathrm{glob}}$ is free, we are done.
\end{proof}

The functor $T$ is compatible with composition of bordisms in an
obvious way, and we will not spell out the definition of the
compositor and unitor promoting $T$ to a functor of internal
categories.

\subsection{On the choice of presentations}
\label{sec:choice-presentations}

We need to argue that $T$, as constructed above, depends only on the
gerbe with connection $\tilde{\mathfrak X} \to \mathfrak X$, and not
on the chosen Lie groupoid presentations for $\mathfrak X$ and
$\tilde{\mathfrak X}$.  To formulate this statement more precisely, we
introduce some notation.  Recall that proposition~\ref{prop:4}
justified the lack of reference to $X_1 \rightrightarrows X_0$ in the
notations $1\vert1\EBord(\mathfrak X)^{\mathrm{skel}}$ and
$1\vert1\EBord(\mathfrak X)^{\mathrm{glob}}$.  In this subsection, we
must be explicit about the choices of presentations, so we will write
\begin{equation*}
  1\vert1\EBord(X_\bullet) = 1\vert1\EBord(\mathfrak X)^{\mathrm{glob}}.
\end{equation*}

Then, what we have achieved in section~\ref{sec:gerbes-1vert1-twists}
is the construction, from a Lie groupoid $X_1 \rightrightarrows X_0$
and a central extension $L \to X_1$, of a functor of internal
categories
\begin{equation*}
  T_L\colon 1\vert1\EBord(X_\bullet) \to \mathrm{Alg}.
\end{equation*}
Suppose now that $X_1' \rightrightarrows X_0'$ and the central
extension $L' \to X_1'$ provide a second presentation of $\mathfrak X$
and the gerbe $\tilde{\mathfrak X}$.  By
\cite[proposition~4.15]{MR2817778}, there exists a Lie groupoid $X_1''
\rightrightarrows X_0''$ and a central extension $L'' \to X_1''$
together with a Morita morphism $X_\bullet'' \to X_\bullet$ and a
Morita morphism of central extensions $L'' \to L$, as well as similar
data for $L' \to X_1' \rightrightarrows X_0'$.  These Morita morphisms
are uniquely determined, up to unique natural isomorphism, by the
requirement that they induce the identity of $\tilde{\mathfrak X} \to
\mathfrak X$ \cite[section~2.6]{MR2817778}.

We now have a diagram as follows.
\begin{equation}
  \label{eq:22}
  \begin{gathered}
  \xymatrix{
    &
    1\vert1\EBord(X_\bullet'')
    \ar[ld] \ar[rd] \ar[dd]^{T_{L''}}
    & \\
    1\vert1\EBord(X_\bullet')
    \ar[rd]_{T_{L'}}
    &&
    1\vert1\EBord(X_\bullet)
    \ar[ld]^{T_L}
    \\ &
    \mathrm{Alg}
    &
  }
  \end{gathered}
\end{equation}
The claim that $T_L$ determines a twist functor $T =
T_{\tilde{\mathfrak X}}$ depending only on the gerbe $\tilde{\mathfrak
  X}$ over $\mathfrak X$ is formalized by the following statement.

\begin{proposition}
  In the above situation, there exist canonical $2$-morphisms making
  \eqref{eq:22} commute.
\end{proposition}

The proof is just a verification that the construction of $T_L$ is
natural with respect to internal functors between Lie groupoids (of
which Morita morphisms are particular cases), so we will omit further
details.

\begin{remark}
  Not every gerbe $\tilde{\mathfrak X}$ over $\mathfrak X$ is
  necessarily presentable as a central extension of a given
  presentation $X_1 \rightrightarrows X_0$; this condition is
  equivalent to $\tilde{\mathfrak X}$ admitting a trivialization when
  restricted to $X_0$ \cite[proposition~4.12]{MR2817778}.  Thus, given
  $\tilde{\mathfrak X}$, we need to pick a sufficiently fine groupoid
  presentation of $\mathfrak X$.  In this paper, we never let
  $\tilde{\mathfrak X}$ vary or make structural statements about the
  groupoid of twists $1\vert1\ETw(\mathfrak X)$, so we are allowed to
  fix, once and for all, a Lie groupoid $X_1 \rightrightarrows X_0$
  suitable for $\tilde{\mathfrak X}$.
\end{remark}

\subsection{The restriction to $\mathfrak K(\mathfrak X)$}
\label{sec:restriction-to-K(X)}

We denote by $\mathfrak K(\mathfrak X)$ the substack of closed and
connected bordisms in $1\vert1\EBord(\mathfrak X)^{\mathrm{glob}}$.
The twist functor $T_{\tilde{\mathfrak X}}$ determines, by
restriction, a line bundle $Q$ on $\mathfrak K(\mathfrak X)$.  Our
goal in this section is to give a detailed description of $Q$, in
terms of a Čech cocycle for Deligne cohomology representing the gerbe
$\tilde{\mathfrak X} \to \mathfrak X$.

Let us start fixing some notation.  The orbifold $\mathfrak X$ will be
presented, as before, by an étale Lie groupoid $X_1 \rightrightarrows
X_0$, and the gerbe $\tilde{\mathfrak X} \to \mathfrak X$ will be
presented by a Čech 2-cocycle for groupoid cohomology with
coefficients in the Deligne complex $\mathbb C^\times(3)$,
\begin{equation*}
  \mathbb C^\times \xrightarrow{d\log} \Omega^1 \to \Omega^2.
\end{equation*}
More explicitly, this cocycle is given by a triple
\begin{equation*}
  (h, A, B) \in C^\infty(X_2, \mathbb C^\times) \times
  \Omega^1_{\mathbb C}(X_1) \times \Omega^2_{\mathbb C}(X_0)
\end{equation*}
satisfying the cocycle conditions
\begin{gather}
  \label{eq:9}
  h(a,b) h(a,bc)^{-1} h(ab,c) h(b,c)^{-1} = 1 \text{ in }
  C^\infty(X_3, \mathbb C^\times), \\
  \label{eq:10}
  \pr_2^*A + \pr_1^*A - c^*A = d\log h \text{ in } \Omega^1(X_2),\\
  \label{eq:11}
  t^*B - s^*B = dA \text{ in } \Omega^2(X_1),
\end{gather}
where $X_n = X_1 \times_{X_0} \cdots \times_{X_0} X_1$ is the space of
sequences of $n$ composable morphisms.

To an object $(\psi,I) \in \mathfrak K(\mathfrak X)_S$ as above the
fibered functor
\begin{equation*}
  Q = T_{\tilde{\mathfrak X}}\vert_{\mathfrak K(\mathfrak X)}
  \colon \mathfrak K(\mathfrak X) \to \mathrm{Vect}
\end{equation*}
assigns the trivial line bundle over $S$; the interesting discussion,
of course, concerns morphisms.  Fix a second object $(\psi',I') \in
\mathfrak K(\mathfrak X)_{S'}$ and a morphism $\lambda\colon (\psi,I)
\to (\psi',I')$ over $f\colon S \to S'$.  To that $Q$ assigns a linear
map between the corresponding lines, which we identify with a function
$Q(\lambda)\colon S \to \mathbb C^\times$.  We consider the two
special cases of section~\ref{sec:gerbes-1vert1-twists}, using the
notation fixed there.

\begin{proposition}
  If $\lambda$ is a refinement of skeletons, then $Q(\lambda) = 1$.
  If the skeletons of $K$, $K'$ are compatible, then $Q(\lambda) \in
  C^\infty(S, \mathbb C^\times)$ is given by
  \begin{equation}
    \label{eq:3}
    Q(\lambda) =
    \exp \left(
      \sum_{1\leq i\leq n}\int_{J_i}\vol_D\, \langle D, \psi_1^*A\rangle 
    \right)
    \prod_{1\leq i\leq n} \left.
      \frac{\psi_2^*h(\tilde a_i, j_i)}{\psi_2^*h(\lambda^*j_i', \tilde b_{i-1})}
    \right\vert_{S\times\{ 0 \}}.
  \end{equation}
\end{proposition}
Here, $D \in C^\infty(TU)$ is a choice of Euclidean vector field for
the Euclidean structure induced by $\pi\colon U \to K$, and $\vol_D$
the corresponding volume form (cf.\
appendix~\ref{sec:fund-theor-calc}).  Moreover, $\psi_2\colon U
\times_KU\times_KU \to X_2$ denotes the map induced by $\psi_1$.

\begin{proof}
  The first claim is obvious.  As to the second claim, each $h_i$ in
  \eqref{eq:2} contributes one factor in the product, and each
  $\mathrm{SP}_i$ contributes one term in the summation.  In fact, we
  see easily from proposition \ref{prop:10} that super parallel
  transport on along a superinterval $J$ with respect to the
  connection form $d - A$ is given by $\exp(\int_J \vol_D \langle
  D,A\rangle)$.  All terms of the from $h(f,f^{-1})$ introduced by
  identifications $L_f^\vee \cong L_{f^{-1}}$ cancel out.
\end{proof}

\begin{remark}
  From the data of a Čech-cocycle presentation of a gerbe with
  connection, \textcite{MR2240409} constructed a line bundle with
  connection on the loop orbifold $L\mathfrak X$.  Our construction
  incorporates a super analogue of this transgression procedure:
  compare \eqref{eq:3} with their definition~4.2.  Our proof that $T$
  is a functor (or, rather, its purely bosonic analog) provides a more
  geometric explanation for Lupercio and Uribe's calculations with
  Deligne Čech cocycles.
\end{remark}

\begin{remark}
  The usual transgression of gerbes produces, in fact, a
  $\Diff^+(S^1)$-equivariant line bundle on the loop space
  \cite[proposition 6.2.3]{MR2362847}.  Our construction gives a line
  bundle on the moduli stack of supercircles over $\mathfrak X$, and
  not just on a ``super loop space'', so the super analogue of
  $\Diff^+(S^1)$-equivariance is automatically built into our
  discussion.
\end{remark}

\section{Dimensional reduction of twists}
\label{sec:dimens-reduct-twists}

In this section, we show that dimensional reduction of the
$1\vert1$-twists from section~\ref{sec:twists} recovers the
$0\vert1$-twists described in section~\ref{sec:tu-xu-de-rham}.  So we
start with the twist functor
\begin{equation*}
  T_{\tilde{\mathfrak X}} \in 1\vert1\ETw(\mathfrak X)
\end{equation*}
associated to a gerbe with connection $\tilde{\mathfrak X}$, and
describe its pullback to $0\vert1\EBord^{\mathbb T}(\Lambda\mathfrak
X)$ by the functor in \eqref{eq:21}.  As we will see below, that
corresponds to the data of a line bundle with superconnection on
$\Lambda\mathfrak X$.  We will then find that it is flat, and hence,
by proposition~\ref{prop:3}, descends to a line bundle
$T'_{\tilde{\mathfrak X}}$ on $\mathfrak B (\Lambda\mathfrak X) = \Pi
T\Lambda\mathfrak X\sslash\Isom(\mathbb R^{0\vert1})$, or, in other
words, a $0\vert1$-dimensional Euclidean twist over $\Lambda\mathfrak
X$.  We call $T'_{\tilde{\mathfrak X}}$ the dimensional reduction of
$T_{\tilde{\mathfrak X}}$.

\begin{theorem}
  \label{thm:3}
  Let $\mathfrak X$ be an orbifold, $\tilde{\mathfrak X}$ a gerbe with
  connection, and $\alpha \in H^3(\mathfrak X;\mathbb Z)$ its
  Dixmier-Douady class.  Then the twist $T'_{\tilde{\mathfrak X}} \in
  0\vert1\ETw(\Lambda\mathfrak X)$ obtained by dimensional reduction
  of the twist $T_{\tilde{\mathfrak X}} \in 1\vert1\ETw(\mathfrak X)$
  is isomorphic to the twist $T_\alpha$ from theorem~\ref{thm:2}.
\end{theorem}

This means, in particular, that
\begin{equation*}
  0\vert1\EFT^{T'_{\tilde{\mathfrak X}}\otimes T_i}[\Lambda\mathfrak X] \cong
  K^{i+\alpha}(\mathfrak X) \otimes \mathbb C,
\end{equation*}
and suggests that $T_{\tilde{\mathfrak X}}$ is the correct
$1\vert1$-twist to represent $\alpha$-twisted $K$-theory in the sense
of \eqref{eq:18}.

The proof of the theorem (and the flatness claim necessary to state
it) will occupy the remainder of this section.  Before getting
started, we record a technical lemma.

\begin{lemma}
  \label{lem:2}
  Let $X$ be an ordinary manifold, $\mathrm{ev}\colon \Pi TX \times
  \mathbb R^{0\vert1} \to X$ be evaluation map, and write $\tilde
  \omega$ for the function on $\Pi TX$ corresponding to the
  differential form $\omega \in \Omega^n(X)$.  Then we have
  \begin{equation*}
    \langle (\partial_\theta)^{\wedge n}, \mathrm{ev}^* \omega\rangle =
    \pm n!(\tilde\omega + \theta \widetilde{d\omega}),
  \end{equation*}
  where the sign $\pm$ is $-1$ if $n \equiv 2, 3$ mod $4$ and $+1$
  otherwise.
\end{lemma}

\begin{proof}
  It suffices to prove the lemma for $\omega = f_0df_1\cdots df_n$,
  where $f_i\in C^\infty(X)$.  The action $\mu\colon \Pi TX \times
  \mathbb R^{0\vert1} \to \Pi TX$ is given by the formula
  \begin{equation*}
    \mu^\sharp\colon \tilde \omega
    \mapsto \tilde \omega + \theta D\tilde{\omega},
  \end{equation*}
  where $D$ denotes the de Rham vector field on $\Pi TX$.  Thus
  \begin{align*}
    \mathrm{ev}^*\omega
    &= (f_0 + \theta Df_0) \prod_{1\leq i\leq n}
      d(f_i+\theta Df_i) \\
    &= (f_0 + \theta Df_0) \prod_{1\leq i\leq n}
      (df_i+ d\theta Df_i +\theta dDf_i).
  \end{align*}
  Writing $F_i = df_i+ d\theta Df_i +\theta dDf_i$, we have
  \begin{equation*}
    i_{\partial_\theta}\mathrm{ev}^*\omega = (-1)^{i-1} \sum_{1\leq i\leq
        n}  (f_0+\theta Df_0) F_1 \cdots F_{i-1} Df_i F_{i+1} \cdots F_n.
  \end{equation*}
  Contracting an expression like the one under the summation sign with
  $\partial_\theta$ produces as many new terms as there are $F_i$'s,
  and, in each of those, one of the $F_i$'s get converted into a
  $Df_i$.  Note also that commuting $i_{\partial_\theta}$ with any
  $F_i$ or $Df_i$ produces a minus sign.  Iterating this process, we
  find
  \begin{align*}
    (i_{\partial_\theta})^n\mathrm{ev}^*\omega
    & = (-1)^{0+1+\cdots+n-1} n! (f_0+\theta Df_0) Df_1 \cdots Df_n\\
    &= \pm n!(\tilde\omega + \theta \widetilde{d\omega}). \qedhere
  \end{align*}
\end{proof}

\begin{corollary}
  \label{cor:1}
  For $D$ the de Rham vector field on $\Pi TX$, $\mu\colon \Pi TX
  \times \mathbb R^{0\vert1} \to \Pi TX$ the action map, $\pi\colon
  \Pi TX \to X$ the projection, and the remaining notation as in the
  lemma,
  \begin{equation*}
    \mu^*\langle D^{\wedge n}, \pi^*\omega\rangle =
    \pm n!(\tilde\omega + \theta \widetilde{d\omega}).
  \end{equation*}
\end{corollary}

\subsection{Review of dimensional reduction}
\label{sec:remind-dimens-reduct}

In this subsection we recall the main points about our dimensional
reduction procedure, also fixing the notation.  See
\cite{arXiv:1703.00314} for the complete story.

It is sufficient to describe the effect of the internal functors
\eqref{eq:21} on the corresponding stacks of closed and connected
bordisms, which we denote
\begin{equation}
  \label{eq:24}
  \mathfrak B(\Lambda\mathfrak X)
  \xleftarrow{\mathcal L} \mathfrak B^{\mathbb T}(\Lambda\mathfrak X)
  \xrightarrow{\mathcal R} \mathfrak K(\mathfrak X).
\end{equation}
The left stack was already introduced in
section~\ref{sec:twisted-de-rham}, and is given by
\begin{equation*}
  \mathfrak B(\Lambda\mathfrak X) = \underline{\Fun}_{\mathrm{SM}}(\mathbb
  R^{0\vert1}, \Lambda\mathfrak X)\sslash\mathrm{Isom}(\mathbb R^{0\vert1}).
\end{equation*}
The middle stack is defined as
\begin{equation*}
  \mathfrak B^{\mathbb T}(\Lambda\mathfrak X) = \underline{\Fun}_{\mathrm{SM}}(\mathbb
  R^{0\vert1}, \Lambda\mathfrak X)\sslash\mathrm{Isom}(\mathbb R^{1\vert1}),
\end{equation*}
and the map $\mathcal L$ is induced by the group homomorphism
\begin{equation*}
  \Isom(\mathbb R^{1\vert1}) = \mathbb R^{1\vert1} \rtimes \mathbb Z/2
  \to \mathbb R^{0\vert1} \rtimes \mathbb Z/2 = \Isom(\mathbb R^{0\vert1}).
\end{equation*}

Next we turn to a brief description of the map $\mathcal R$, focusing
on our case of interest.  We fix an étale Lie groupoid presentation
$X_1 \rightrightarrows X_0$ for $\mathfrak X$, so that we also get
presentations
\begin{equation*}
  \Pi T X_1 \rightrightarrows \Pi T X_0, \quad
  \hat X_1 \rightrightarrows \hat X_0, \quad
  \Pi T\hat X_1 \rightrightarrows \Pi T\hat X_0,
\end{equation*}
of $\Pi T\mathfrak X$, $\Lambda\mathfrak X$ and $\Pi T\Lambda\mathfrak
X$ respectively.  Then we have an atlas
\begin{equation*}
  \check x\colon \Pi T\hat X_0 \to \mathfrak B^{\mathbb T}(\Lambda \mathfrak X).
\end{equation*}
Now, we want to describe the $\Pi T\hat X_0$-family classified by the
map $\mathcal R \circ \check x$, which we will denote $K_{\check x}$.
Chasing through the construction of \cite{arXiv:1703.00314}, we find
that
\begin{equation*}
  K_{\check x} = (K, \psi\colon K \to \mathfrak X, I)
  \in \mathfrak K(\mathfrak X)
\end{equation*}
is given, in the language of torsors, by the following data:
\begin{enumerate}
\item the trivial family $K = \Pi T \hat X_0 \times \mathbb
  T^{1\vert1}$ of length $1$ supercircles, together with the standard
  covering $U = \Pi T\hat X_0 \times \mathbb R^{1\vert1}\to K$,
\item the map $\psi_0\colon U \to X_0$ given by the composition
  \begin{equation*}
    U = \Pi T\hat X_0 \times \mathbb R^{1\vert1}
    \to \Pi T\hat X_0 \times \mathbb R^{0\vert1}
    \xrightarrow{\mathrm{ev}} \hat X_0
    \overset p \twoheadrightarrow X_0,
  \end{equation*}
\item the map $\psi_1\colon U \times_K U \to X_1$ which, over the
  component of points differing by $n$ units, is the $n$-fold iterate
  of
  \begin{equation*}
    \alpha\colon \Pi T\hat X_0 \times \mathbb R^{1\vert1} \to \Pi T\hat X_0
    \times \mathbb R^{0\vert1} \xrightarrow{\mathrm{ev}} \hat X_0
    \overset i \hookrightarrow X_1,
  \end{equation*}
\item a skeleton we may choose to be $\Pi T\hat X_0 \times
  [0,1]\subset U$.
\end{enumerate}

\begin{remark}
  To arrive at the above description of $K_{\check x}$, it is helpful
  to note that $\mathfrak B^{\mathbb T}(\Lambda\mathfrak X)$ admits a more
  geometrical formulation (cf.\ \cite[section~3.2]{arXiv:1703.00314}),
  where the map $\check x$ classifies the $\Pi T\hat X_0$-family of
  gadgets given by
  \begin{enumerate}
  \item the trivial family of Euclidean $0\vert1$-manifolds $\Sigma =
    \Pi T\hat X_0 \times \mathbb R^{0\vert1} \to \Pi T\hat X_0$,
  \item the trivial $\mathbb T$-bundle $P = \Pi T\hat X_0 \times
    \mathbb T^{1\vert1} \to \Sigma$ with the standard principal
    connection, and
  \item the map $\Sigma \to \Lambda\mathfrak X$ given by the
    composition
    \begin{equation*}
      \Sigma = \Pi T\hat X_0 \times \mathbb R^{0\vert1}
      \xrightarrow{\mathrm{ev}} \hat X_0
      \xrightarrow{\hat x} \Lambda\mathfrak X,
    \end{equation*}
    where $\hat x$ is the versal family for $\Lambda\mathfrak X$.
  \end{enumerate}
  In this picture, the map $\mathcal L$ simply forgets $P$.  The
  subtle aspect about $\mathcal R$ is that the map $\psi\colon K \to
  \mathfrak X$ is not simply the composition
  \begin{equation*}
    K = P \to \Sigma \to \Lambda\mathfrak X \to \mathfrak X,
  \end{equation*}
  but, rather, is given by a descent construction using the canonical
  automorphism of the inertia $\Lambda\mathfrak X$.
\end{remark}

Eventually, we will need to understand the action of $\mathbb
R^{1\vert1} \subset \Isom(\mathbb R^{1\vert1})$ on $K_{\check x}$ by
rotations.  More precisely, the natural isomorphism between the stack
maps
\begin{equation*}
  \Pi T\hat X_0 \times \mathbb R^{1\vert1}
  \underset\mu{\overset\pr\rightrightarrows} \Pi T\hat X_0
  \xrightarrow{\check x} \mathfrak B^{\mathbb T}(\Lambda\mathfrak X)
\end{equation*}
leads to an isomorphism
\begin{equation*}
  \mu^*K_{\check x} \cong \pr^*K_{\check x}
  \text{ over }
  \Pi T\hat X_0 \times \mathbb R^{1\vert1}.
\end{equation*}
This will be described later (cf.\ figure~\ref{fig:step-by-step}), but
we would like to notice two useful facts now.  First, $\mathcal R$ has
image in the substack
\begin{equation*}
  \mathfrak K_1(\mathfrak X) \cong \underline\Fun_{\mathrm{SM}}(\mathbb
  T^{1\vert1}, \mathfrak X)\sslash \Isom(\mathbb T^{1\vert1})
\end{equation*}
of length $1$ supercircles, so that the $\mathbb R^{1\vert1}$-action
comes from the action on $\mathbb T^{1\vert1}$ by rotations.  Second,
an expression of $K_{\check x}$ as a composition of more basic
bordisms (left and right elbows and braiding) is obtained as follows.

Let $R_{\check x}^{[0,r]}$ be the $S = (\Pi T\hat X_0 \times \mathbb
R^{1\vert1}_{>0})$-family of bordisms such that
\begin{enumerate}
\item its image in $1\vert1\EBord_1$ is simply the pullback of the
  $\mathbb R^{1\vert1}_{>0}$-family $R_r$ described in
  section~\ref{sec:generators-relations};
\item the $X$-torsor over $\Sigma = \Pi T\hat X_0 \times \mathbb
  R^{1\vert1}_{>0} \times \mathbb R^{1\vert1}$ is given by the trivial covering
  \begin{equation*}
    U = \Pi T\hat X_0 \times \mathbb
  R^{1\vert1}_{>0} \times \mathbb R^{1\vert1} \to \Sigma,
  \end{equation*}
  the map
  \begin{equation*}
    \psi_0\colon
    U = \Pi T\hat X_0 \times \mathbb R^{1\vert1}_{>0} \times \mathbb R^{1\vert1}
    \xrightarrow{\pr} \Pi T\hat X_0 \times \mathbb R^{0\vert1}
    \xrightarrow{\mathrm{ev}} \hat X_0
    \overset p \twoheadrightarrow X_0,
  \end{equation*}    
  and the obvious $\psi_1\colon U \times_\Sigma U = U \to X_1$;
\item the skeleton is given by the family $[0,r] \subset \Sigma$ of
  superintervals, where $r$ denotes the standard coordinate function of
  the $\mathbb R^{1\vert1}_{>0}$ factor of $S$.
\end{enumerate}
For future reference, note that there are obvious variants
$R^{[r,1]}_{\check x}$, $R^{[r,1+r]}_{\check x}$ corresponding to
different choices of skeletons.  Also, we write $R^{[0,1]}_{\check x}$
for the restriction of $R^{[0,r]}_{\check x}$ to $\Pi T\hat X_0 \times
\{ 1 \}$.  This description also fixes $\Pi T\hat X_0$-families
$\mathrm{sp}_{\check x}, \overline{\mathrm{sp}}_{\check x} \in
1\vert1\EBord(\mathfrak X)_0$ such that $R_{\check x}^{[0,1]}$ is a
bordism $\emptyset \to \mathrm{sp}_{\check x} \amalg
\overline{\mathrm{sp}}_{\check x}$, and therefore also a left elbow
$L_{\check x}\colon \overline{\mathrm{sp}}_{\check x} \amalg
\mathrm{sp}_{\check x} \to \emptyset$.  With this notation, we can
finally write
\begin{equation*}
  K_{\check x}
  \cong L_{\check x}
  \circ \sigma_{\mathrm{sp}_{\check x}, \overline{\mathrm{sp}}_{\check x}}
  \circ R^{[0,1]}_{\check x}.
\end{equation*}

\subsection{The underlying line bundle}
\label{sec:underly-line-bundle}

Our goal for the remainder of this proof is to understand the
restriction of the line bundle $Q$ from
section~\ref{sec:restriction-to-K(X)}, which we call
\begin{equation*}
  Q'\colon \Pi T\Lambda\mathfrak X\sslash \Isom(\mathbb R^{1\vert1})
  = \mathfrak B^{\mathbb T}(\Lambda\mathfrak X)
  \xrightarrow{\mathcal R} \mathfrak K(\mathfrak X)
  \xrightarrow{Q} \mathrm{Vect}.
\end{equation*}
Thus, $Q'$ is identified with line bundle with superconnection on
$\Lambda\mathfrak X$ (cf.\ remark~\ref{rem:nonflat}).

Let us not worry about the superconnection for now and simply describe
the underlying line bundle.  Thus our goal is to describe the line
bundle on $\Pi T\hat X_0$ (which we still call $Q'$) induced by the
atlas $\check x$ and the isomorphism $H$ between its two pullbacks via
the source and target maps $\Pi T\hat X_1\rightrightarrows \Pi T\hat
X_0$.

As a warm-up, pick a point $(x \in X_0, g\colon x \to x) \in \Pi T\hat
X_0$.  It determines a point of $\mathfrak K(\mathfrak X)$ consisting
of the length $1$ constant superpath $x$ in $\mathfrak X$ with its
endpoints glued together via $g$.  Thus $Q'_{(x,g)} = L_g$.  From this
we already conclude that $Q' \to \Pi T\hat X_0$ is trivial, since we
assume the same of $L$.  Now fix a second point $(x',g') \in \Pi T\hat
X_0$ and a compatible morphism $f\colon x \to x'$; meaning that $g' =
fgf^{-1}$.  This gives rise to a morphism $\bar f$ in $\mathfrak
K(\mathfrak X)$ which, in the same vein as above, we can describe as
being the constant $f$.  From \eqref{eq:3}, we conclude that the
identification $H(f)\colon Q'_{(x,g)} \to Q'_{(x',g')}$ is through
\begin{equation*}
  H(f) = L(\bar f)
  = \frac{h(f,g)}{h(g',f)}
  = \frac{h(f,g)}{h(fgf^{-1},f)},
\end{equation*}
which agrees with \eqref{eq:14}.

Since we worked above with \emph{points}, we have only shown that the
restriction of $Q'$ to $\Lambda\mathfrak X \hookrightarrow \Pi
T\Lambda\mathfrak X$ agrees with the $L'$ from
section~\ref{sec:tu-xu-de-rham}.  We want a slightly stronger
statement, namely we want to identify $Q'$ with the a pullback of $L'$
via $\Pi T\Lambda\mathfrak X \to \Lambda\mathfrak X$.  To do this, we
need to fully describe the isomorphism $H$, which we see as a function
$H\colon \Pi T\hat X_1 \to \mathbb C^\times$.  Consider again the
versal family $\check x\colon \Pi T\hat X_0 \to \mathfrak B^{\mathbb
  T}(\mathfrak X)$.  Then $H$ is determined by the composition
\begin{equation*}
  \xymatrix@C=1.5em@R=1ex{
    & \Pi T\hat X_0 \ar@{=>}[dd] \ar[dr]^{\check x}\\
    \Pi T\hat X_1 \ar[ur]^s \ar[dr]_t
    && \mathfrak B^{\mathbb T}(\Lambda \mathfrak X) \ar[r]^{\mathcal R}&
    \mathfrak K(\mathfrak X)
    \ar[r]^U & \mathrm{Vect}.\\
    & \Pi T\hat X_0  \ar[ur]_{\check x}}
\end{equation*}
From \eqref{eq:3} and the above description of $K_{\check x} =
\mathcal R \circ \check x$, we see that
\begin{equation*}
  H = \exp\left( \int_{[0,1]} \vol_D\,\langle D, \beta^*A
    \rangle\right)
  \frac{h(f,g)}{h(fgf^{-1},f)},
\end{equation*}
where $\beta$ is the composition
\begin{equation*}
  U = \Pi T\hat X_1 \times \mathbb R^{1\vert1}
  \to \Pi T\hat X_1 \times \mathbb R^{0\vert1}
  \xrightarrow{\mathrm{ev}} \hat X_1
  \xrightarrow{p} X_1,
\end{equation*}
the integral is fiberwise over $\Pi T\hat X_1$, and $f$, $g$ are
similar to the paragraph above (except now they stand for $\hat
X_1$-points of $X_1$ instead of $\mathrm{pt}$-points).  From
lemma~\ref{lem:2}, we have
\begin{equation*}
  \int_{[0,1]} \vol_D\,\langle D, \mathrm{ev}^*A \rangle
  = \int_{[0,1]} \vol_D\, \tilde A + \theta \widetilde{dA} =
  \widetilde{dA} = t^*\tilde B - s^*\tilde B,
\end{equation*}
so that
\begin{equation*}
  H = \frac{t^*\exp(\tilde B)}{s^*\exp(\tilde B)}
  \frac{h(f,g)}{h(fgf^{-1},f)}.
\end{equation*}
Seeing $\exp(-\tilde B)\colon \Pi T\hat X_0 \to \mathbb C^\times$ as
an isomorphism
\begin{equation}
  \label{eq:25}
  \exp(-\tilde B)\colon Q' \to \pi^*L',
\end{equation}
of trivial line bundles, $H$ gets identified with the defining datum
\eqref{eq:14} of $L'$, as desired.

\subsection{The superconnection}
\label{sec:superconnection}

A superconnection on the line bundle $L'\colon \Lambda\mathfrak X \to
\mathrm{Vect}$ consists of a superconnection $\mathbb A$ on the
underlying line bundle $L' \to \hat X_0$ whose two pullbacks over
$\hat X_1$ are identified with one another through the isomorphism
$s^*L' \to t^*L'$.  Since we just want to describe the superconnection
on $\hat X_0$ (which we know a priori to be invariant), it suffices to
look at the versal family
\begin{equation*}
  \check x\colon \Pi T\hat X_0
  \to \mathfrak B^{\mathbb T}(\Lambda\mathfrak X) 
  = \Pi T\Lambda\mathfrak X\sslash\mathrm{Isom}(\mathbb R^{1\vert1})
\end{equation*}
and its image in $\mathfrak K(\mathfrak X)$; nothing here will involve
$\hat X_1$.  Now, the superconnection we are seeking to describe is
geometrically encoded by the $\mathbb R^{1\vert1}$-action on $\Pi
T\hat X_0$ and the line bundle $Q'$ over it.  More specifically, the
operator $\mathbb A$ is the infinitesimal generator associated to the
vector field $D = \partial_\theta - \theta \partial_t$.

Thus, we need to understand the action of $\mathbb R^{1\vert1}$ by
rotations of the $\Pi T\hat X_0$-family $K_{\check x}$, i.e., the
isomorphism
\begin{equation}
  \label{eq:8}
  \mu^*K_{\check x} \to \pr^*K_{\check x}
  \text{ over }
  \Pi T\hat X_0 \times \mathbb R^{1\vert1},
\end{equation}
and its image under $T_{\tilde{\mathfrak X}}$, where $\mu, \pr\colon
\Pi T\hat X_0 \times \mathbb R^{1\vert1} \to \Pi T\hat X_0$ are the
action respectively projection maps.

The isomorphism \eqref{eq:8} can be expressed as a composition of
simpler steps as indicated in figure~\ref{fig:step-by-step}.  There,
the left elbows
\begin{tikzpicture}[y=.7em,baseline=0pt]
  \draw[l,L,looseness=2] (0,1) to (0,0);
  \draw (0,0) pic {dotm} (0,1) pic {dotp};
\end{tikzpicture}
always represent the bordism $L_{\check x}$, or, more precisely,
its pullback via $\pr\colon \Pi T\hat X_0 \times \mathbb R^{1\vert1}
\to \Pi T\hat X_0$; the right elbows
\begin{tikzpicture}[y=.7em,baseline=0pt]
  \draw[r,R,looseness=2] (0,1) to["{$\scriptstyle[a,b]$}"{right, inner sep=.1em}] (0,0);
  \draw (0,0) pic {dotp} (0,1) pic {dotm};
\end{tikzpicture}
represent bordisms $R_{\check x}^{[a,b]}$ with the indicated skeleton
$[a,b]$; straight and crossing lines denote appropriate identities and
braidings or, more precisely, their avatars as thin bordisms.  Thus,
for instance, the second picture represents the composition
\begin{equation*}
  L_{\check x}
  \circ \sigma_{\mathrm{sp}_{\check x} \amalg \overline{\mathrm{sp}}_{\check x}}
  \circ (\mathrm{Id}_{\overline{\mathrm{sp}}_{\check x}}
  \amalg L_{\check x}
  \amalg \mathrm{Id}_{\mathrm{sp}_{\check x}})
  \circ (R^{[0,r]}_{\check x} \amalg R^{[r,1]}_{\check x})
\end{equation*}
(leaving implicit, as usual, pullbacks via projection maps).  The
isometries between successive pictures are the obvious ones.  For
example, the second isomorphism is the semigroup property of right
elbows and the fourth uses the symmetry condition of the braiding.
Now, the image of each of the intermediate steps under
$T_{\tilde{\mathfrak X}}$ is a canonically trivial line bundle (over
$\Pi T\hat X_0 \times \mathbb R^{1\vert1}$), and, with the exception
of the fifth step, the corresponding isomorphism of line bundles is
the identity.  In fact, under the canonical trivilizations of all line
bundles in question, we have an identity
\begin{equation*}
  T_{\tilde{\mathfrak X}}(\mu^*K_{\check x} \to \pr^*K_{\check x}) =
  T_{\tilde{\mathfrak X}}(R_{\check x}^{[1,1+r]} \to R_{\check x}^{[0,r]}).
\end{equation*}
The right-hand side, once identified with a $\mathbb C$-valued
function on $\Pi T\hat X_0$, can be calculated from \eqref{eq:3} and
is given by
\begin{equation*}
  \exp\left(-\int_{[0,r]} \vol_D\,\langle D, \mathrm{ev}^*A \rangle \right)
  = \exp\int_{[0,r]} \vol_D\, - \tilde A =
  \exp(\theta \tilde A) = 1 + \theta \tilde A.
\end{equation*}
Here, the second equality uses corollary~\ref{cor:1} and the fact that
$\widetilde{dA}$ vanishes, the integral being fibered over $\Pi T\hat
X_0$.

\begin{figure}
  \def\Rsup#1{\scriptstyle #1}
  \def\extraspace{2em}
  \centering
  \begin{tabular}{rcl}
  $\pr^*K_{\check x}$ & $\cong$ & 
  \begin{tikzpicture}[baseline={(0,0.5)}]
    \draw[l,L] (0,1) to (0,0);
    \draw[S] ([r] 0,0) to ([l] 1,1);
    \draw[S, preaction={draw,ultra thick,white}] ([r] 0,1) to ([l] 1,0);
    \draw[r,R] (1,0)
      to[looseness=3, "$\Rsup{[0,1]}$"right] (1,1);
    \draw (0,0) pic {dotm}
          (0,1) pic {dotp}
          (1,0) pic {dotp}
          (1,1) pic {dotm};
   \end{tikzpicture}
   $\quad\cong\quad$     
  \begin{tikzpicture}[baseline={(0,0.5)}]
    \draw[l, L] (0,1) to (0,0);
    \draw[S] ([r]0,0) to ([l]1,1);
    \draw[S, preaction={draw,ultra thick,white}] ([r]0,1) to ([l]1,0);
    \draw ([r]1,0) to ([l]2,0) ([r]1,1) to ([l]2,1);
    \draw[l,L] (2,.6) to (2,.4);
    \draw[r,R] (2,0) to[looseness=3, "$\Rsup{[0,r]}$"right] (2,.4);
    \draw[r,R] (2,.6) to[looseness=3, "$\Rsup{[r,1]}$"right] (2,1);
    \draw (0,0) pic {dotm}
          (0,1) pic {dotp}
          (1,0) pic {dotp}
          (1,1) pic {dotm}
          (2,0) pic {dotp}
          (2,.4) pic {dotm}
          (2,.6) pic {dotp}
          (2,1) pic {dotm};
   \end{tikzpicture}
  \\[\extraspace] & $\cong$ &
  \begin{tikzpicture}[baseline={(0,0.5)}]
    \draw[l, L] (0,1) to (0,0);
    \draw[S] ([r]0,0) to ([l]1,1);
    \draw[S, cross] ([r]0,1) to ([l]1,0);
    \draw ([r]1,0) to ([l]2,0) ([r]1,1) to ([l]2,1);
    \draw[l,L] (2,.6) to (2,.4);
    \draw[S] ([r]2,0) to ([l]3,.6) ([r]2,.4) to ([l]3,1);
    \draw[S,cross] ([r]2,1) to ([l]3,.4) ([r]2,.6) to ([l]3,0);
    \draw[r,R] (3,0) to[looseness=3, "$\Rsup{[r,1]}$"right] (3,.4);
    \draw[r,R] (3,.6) to[looseness=3, "$\Rsup{[0,r]}$"right] (3,1);
    \draw (0,0) pic {dotm}
          (0,1) pic {dotp}
          (1,0) pic {dotp}
          (1,1) pic {dotm}
          (2,0) pic {dotp}
          (2,.4) pic {dotm}
          (2,.6) pic {dotp}
          (2,1) pic {dotm}
          (3,0) pic {dotp}
          (3,.4) pic {dotm}
          (3,.6) pic {dotp}
          (3,1) pic {dotm};
   \end{tikzpicture}
   $\quad\cong\quad$     
  \begin{tikzpicture}[baseline={(0,0.5)}]
    \draw[l, L] (0,1) to (0,0);
    \draw[S] ([r]0,0) to ([l]1,1);
    \draw[S, preaction={draw,ultra thick,white}] ([r]0,1) to ([l]1,0);
    \draw ([r]1,0) to ([l]2,0) ([r]1,1) to ([l]2,1);
    \draw[l,L] (2,.6) to (2,.4);
    \draw[r,R] (2,0) to[looseness=3, "$\Rsup{[r,1]}$"right] (2,.4);
    \draw[r,R] (2,.6) to[looseness=3, "$\Rsup{[0,r]}$"right] (2,1);
    \draw (0,0) pic {dotm}
          (0,1) pic {dotp}
          (1,0) pic {dotp}
          (1,1) pic {dotm}
          (2,0) pic {dotp}
          (2,.4) pic {dotm}
          (2,.6) pic {dotp}
          (2,1) pic {dotm};
  \end{tikzpicture}
  \\[\extraspace] & $\cong$ &
  \begin{tikzpicture}[baseline={(0,0.5)}]
    \draw[l, L] (0,1) to (0,0);
    \draw[S] ([r]0,0) to ([l]1,1);
    \draw[S, preaction={draw,ultra thick,white}] ([r]0,1) to ([l]1,0);
    \draw ([r]1,0) to ([l]2,0) ([r]1,1) to ([l]2,1);
    \draw[l,L] (2,.6) to (2,.4);
    \draw[r,R] (2,0) to[looseness=3, "$\Rsup{[r,1]}$"right] (2,.4);
    \draw[r,R] (2,.6) to[looseness=3, "$\Rsup{[1,1+r]}$"right] (2,1);
    \draw (0,0) pic {dotm}
          (0,1) pic {dotp}
          (1,0) pic {dotp}
          (1,1) pic {dotm}
          (2,0) pic {dotp}
          (2,.4) pic {dotm}
          (2,.6) pic {dotp}
          (2,1) pic {dotm};
    \end{tikzpicture}
    $\quad\cong\quad$
    \begin{tikzpicture}[baseline={(0,0.5)}]
    \draw[l,L] (0,1) to (0,0);
    \draw[S] ([r] 0,0) to ([l] 1,1);
    \draw[S, preaction={draw,ultra thick,white}] ([r] 0,1) to ([l] 1,0);
    \draw[r,R] (1,0)
      to[looseness=3, "$\Rsup{[r,1+r]}$"right] (1,1);
    \draw (0,0) pic {dotm}
          (0,1) pic {dotp}
          (1,0) pic {dotp}
          (1,1) pic {dotm};
   \end{tikzpicture}
   \\[\extraspace] & $\cong$ & $\mu^*K_{\check x}$
   \end{tabular}
   \caption{The map \eqref{eq:8}, step by step.}    
  \label{fig:step-by-step}
\end{figure}

Thus, the $\mathbb R^{1\vert1}$-action on $Q' = \Pi T\hat X_0
\times\mathbb C$, expressed as an algebra homomorphism
\begin{equation*}
  C^\infty(\Pi T\hat X_0 \times\mathbb C)
  \to C^{\infty}(\Pi T\hat X_0 \times\mathbb C \times \mathbb
  R^{1\vert1}),
\end{equation*}
is characterized by
\begin{equation*}
  \tilde\omega \in C^\infty(\Pi T\hat X_0) \cong \Omega^*(\hat X_0) \mapsto \tilde\omega + \theta \widetilde{d\omega},\quad
  z \in C^\infty(\mathbb C) \mapsto (1 + \theta \tilde A)z,
\end{equation*}
and the infinitesimal action sends
\begin{equation*}
  \tilde \omega z
  \mapsto (\partial_\theta - \theta\partial_t)
          ((\tilde\omega + \theta \widetilde{d\omega})
          (1 + \theta \tilde A)z)
  = (\widetilde{d\omega} + \tilde A\tilde \omega)z 
  = (D_d + \tilde A) \tilde\omega z.
\end{equation*}
The superconnection corresponding to this odd, fiberwise linear vector
field on the total space of $Q'$ is given by the formula $d + A$ and,
applying the gauge transformation \eqref{eq:25}, we find that the
superconnection on $L'$ is given by the formula
\begin{equation*}
  \mathbb A = d + A + dB = d + A + \Omega,
\end{equation*}
which agrees with \eqref{eq:20}.

\section{$1|1$-EFTs and the Chern character of twisted vector bundles}
\label{sec:ch-twisted}

Fix an orbifold $\mathfrak X$, a gerbe with connection
$\tilde{\mathfrak X}$ and an $\tilde{\mathfrak X}$-twisted vector
bundle with connection $\mathfrak V$, with the usual notation of
appendix~\ref{sec:twisted-superconn}.  In this section, we associate
to $\mathfrak V$ a $1\vert1$-dimensional $T = T_{\tilde{\mathfrak
    X}}$-twisted field theory $E =E_{\mathfrak V}$
\begin{equation*}
  \xymatrix@C=4em{
    1\vert1\EBord(\mathfrak X)^{\mathrm{glob}}
    \ar@/^{4ex}/[r]^1_{}="x"
    \ar@/_{4ex}/[r]_T^{}="y"
    \ar@{=>}"x";"y"^{E}
    & \mathrm{Alg}
  }
\end{equation*}
and show that its dimensional reduction provides a geometric
interpretation of the twisted orbifold Chern character.

If $\tilde{\mathfrak X}$ is trivial, and thus $V$ is just a usual
vector bundle over $\mathfrak X$, the basic idea behind the
construction of the field theory $E = E_V \in 1\vert1\EFT(\mathfrak
X)$ is the following.  To a posivitely oriented superpoint of
$\mathfrak X$, specified by a map to the atlas $X \to \mathfrak X$,
\begin{equation*}
  \mathbb R^{0\vert1} \xrightarrow{x} X \to \mathfrak X,
\end{equation*}
we assign the vector space $V_{x(0)}$.  To a superinterval $\Sigma$ as
in \eqref{eq:7}, we assign the super parallel transport along
$\psi^*V$.  Orientation-reversed manifolds map to the dual vector
spaces and maps, and the image of elbows is determined by the duality
pairing.  We need to check that this is consistent with isometries
between bordisms, in particular those of the form \eqref{eq:7}.  But,
indeed, the data of the vector bundle $V$ on $\mathfrak X$ induces a
superconnection-preserving bundle map $\psi'^*V \to (\psi\circ F)^*V$;
this ensures that $E_V$ is an internal functor.

Now, if $V$ is twisted by a nontrivial gerbe, then $\psi'^*V$ and
$(\psi \circ F)^*V$ only become isomorphic after tensoring with an
appropriate line bundle, namely $\xi^*L$.  As we will see, this
deviation from functoriality is expressed by the fact that $E_V$ is a
twisted field theory, and the relevant twist is the
$T_{\tilde{\mathfrak X}}$ from section~\ref{sec:twists}.

\begin{remark}
  The natural transformation $E$ is not invertible, so we must choose
  between the lax or oplax variants.  We make the choice that better
  fits with our conventions for twisted vector bundles.  Now,
  restricting to the moduli stack of closed, connected bordisms
  $\mathfrak K(\mathfrak X)$, the twist $T$ gives us a line bundle,
  and we made the convention that $E$ maps $1\to T$, and not the
  opposite, so that its partition function (i.e., the restriction
  $E\vert_{\mathfrak K(\mathfrak X)}$) determines a section of
  $T\vert_{\mathfrak K(\mathfrak X)}$, and not of its dual.
\end{remark}

\subsection{Construction of the twisted field theory}
\label{sec:constr-1|1-efts}

Unraveling the definition of natural transformation between internal
functors (see \cite[section~5.1]{MR2742432} for a detailed
explanation), we see that, at the level of object stacks, $E$
determines a symmetric monoidal fibered functor
\begin{equation*}
  E\colon 1\vert1\EBord(\mathfrak X)^{\mathrm{glob}}_0 \to \mathrm{Alg}_1.
\end{equation*}
More specifically, $E$ assigns to an $S$-family $Y$ in the domain a
left $T(Y)$-module $E(Y)$; thus, by construction of $T$, $E(Y)$ is
nothing but a vector bundle over $S$.  When
\begin{equation*}
  Y = (Y \to S,\psi, \iota\colon S\times\mathbb R^{0\vert1} \to U)
\end{equation*}
is a positively oriented superpoint, we set
\begin{equation*}
  E(Y) = V_\iota,
\end{equation*}
where, for any $f\colon S\times\mathbb R^{0\vert1} \to U$, we write
\begin{equation*}
  V_f =  (\psi_0 \circ f)^*V_{\vert S \times 0}.
\end{equation*}
If $Y$ is negatively oriented, then $E(Y)$ is the dual of the above.
To a morphism $\lambda\colon Y \to Y'$ over $f\colon S \to S'$, we
assign the obvious identification
\begin{equation*}
  E(\lambda)\colon V_\iota \to f^*V_{\iota'},
\end{equation*}
which makes sense in view of condition \eqref{eq:17}.  All the
remaining data is determined, uniquely up to unique isomorphism, by
the symmetric monoidal requirement for $E$.

At the level of morphism stacks, $E$ assigns, to each $S$-family of
bordisms $\Sigma$ from $Y_0$ to $Y_1$, a map of $T(Y_1)$-modules;
using the fact that the only algebra in sight is $\mathbb C$, we find
that $E(\Sigma)$ is a linear map
\begin{equation*}
  E(\Sigma)\colon E(Y_1) \to T(\Sigma) \otimes E(Y_0) 
\end{equation*}
of vector bundles over $S$.  Now there are several kinds of bordisms
to consider.  For the sake of clarity, we focus on the superpath stack
$\mathfrak P(\mathfrak X) \hookrightarrow 1\vert1\EBord(\mathfrak
X)_1^{\mathrm{glob}}$.

Fix $\Sigma \in \mathfrak P(\mathfrak X)$, and recall our usual
notation fixed in sections~\ref{sec:skeletons} and
\ref{sec:gerbes-1vert1-twists}.  Instead of describing $E(\Sigma)$, it
is more convenient to describe its inverse
\begin{equation*}
  E^{-1}(\Sigma)\colon \left( \bigotimes\nolimits_{1\leq i\leq n} L_{j_i} \right)
  \otimes V_{a_0} \to V_{b_n},
\end{equation*}
which we set to be the composition
\begin{equation*}
  E^{-1}(\Sigma) = (\id_n \otimes E_n) \circ \dots \circ (\id_1 \otimes E_1) \circ (\id_0 \otimes E_0)
\end{equation*}
where $\id_k$ denotes the identity map of $\bigotimes\nolimits_{k <
  i\leq n} L_{j_i}$, $E_0 = \mathrm{SP}_0$, and $E_i$ is the
composition
\begin{equation}
  \label{eq:4}
  E_i\colon L_{j_i} \otimes V_{b_{i-1}} 
  \xrightarrow{m_{j_i, b_{i-1}}} V_{a_i} \xrightarrow{\mathrm{SP}_i}
  V_{b_i}
\end{equation}
for $1\leq i\leq n$.  Here, $\mathrm{SP}_i$ denotes the super parallel
transport of $\psi_0^*V$ along $I_i$.

Finally, there are certain conditions on the above data that need to
be verified.  It is clear that $Y \mapsto E(Y)$ is a fibered functor,
so it remains to verify that $\Sigma \mapsto E(\Sigma)$ determines a
natural transformation between appropriate fibered functors, and
moreover that this is compatible with compositions of bordisms and
identity bordisms, that is, commutativity of diagrams (3.5) and (3.6)
in \cite{MR2264737}.

\begin{proposition}
  The assignment $\Sigma \mapsto E(\Sigma)$ respects compositions of
  intervals.  That is, given bordisms
  \begin{equation*}
    Y_0 \overset{\Sigma_1}\longrightarrow Y_1 \overset{\Sigma_2}\longrightarrow Y_2 
  \end{equation*}
  in $\mathfrak P(\mathfrak X)$ and writing $\Sigma = \Sigma_2
  \amalg_{Y_1}\Sigma_1$, the diagrams
  \begin{equation*}
    \xymatrix{
      T(\Sigma_2) \otimes T(\Sigma_1) \otimes E(Y_0)
      \ar[d]^{\mu_{\Sigma_2,\Sigma_1}} \ar[r]^-{E^{-1}(\Sigma_1)}
        & T(\Sigma_2) \otimes E(Y_1)\ar[d]^{E^{-1}(\Sigma_2)}\\
      T(\Sigma) \otimes E(Y'_0) \ar[r]^-{E^{-1}(\Sigma)}
        & E(Y_2)
      }
      \quad
      \xymatrix{T(\Id_{Y_0}) \otimes E(Y_0) \ar[r]^-{E^{-1}(\Id_{Y_0})}  \ar[d]^{\epsilon_{Y_0}}&
      E(Y_0)\\
      \Id_{T(Y_0)} \otimes E(Y_0) \ar[ur]}
  \end{equation*}
  commute.
\end{proposition}

\begin{proof}
  This is clear from the definitions.
\end{proof}

\begin{proposition}
  The assignment $\Sigma \mapsto E(\Sigma)$ is natural in $\Sigma \in
  \mathfrak P(\mathfrak X)$, that is, for each morphism $\lambda\colon
  \Sigma \to \Sigma'$ lying over $f\colon S \to S'$, the diagram
  \begin{equation}
    \label{eq:6}
    \begin{gathered}
      \xymatrix{ T(\Sigma) \otimes E(Y_0) \ar[d]^{T(\lambda) \otimes
          E(\lambda_0)}\ar[r]^-{E^{-1}(\Sigma)} & E(Y_1)\ar[d]^{E(\lambda_1)}\\
        T(\Sigma') \otimes E(Y'_0) \ar[r]^-{E^{-1}(\Sigma')} & E(Y'_1) }
    \end{gathered}
  \end{equation}
  commutes.
\end{proposition}

Here, $\lambda_i\colon Y_i \to Y'_i$ denotes the image of $\lambda$ in
$1\vert1\EBord(\mathfrak X)_0$ via the source and target functors for
$i=0,1$ respectively.

\begin{proof}
  If $\lambda$ is a refinement of skeletons, the claim follows
  immediately from the gluing property of super parallel transport.
  Next, let us consider the case where the skeletons of $\Sigma$,
  $\Sigma'$ are compatible and the base map $f\colon S \to S'$ is the
  identity.  For simplicity, we may assume that the skeletons comprise
  three intervals $I_0, I_1, I_2$, and that the first and last of them
  have length zero.  The general case can be deduced by induction,
  using also the compatibility with compositions.

  In this particular situation, \eqref{eq:6} corresponds to the outer
  square in the following diagram.
  \begin{equation*}
    \xymatrix{
      L_{j_2} \otimes L_{j_1} \otimes V_{b_0} \ar[r]^m
      \ar@{}[dr]|{\mathrm{(a)}}
        & L_{j_2} \otimes V_{a_1} \ar[r]^{\mathrm{SP}_V}
        \ar@{}[dr]|{\mathrm{(d)}}
          & L_{j_2} \otimes V_{b_1} \ar[r]^m
            & V_{a_2} \ar@{=}[ddd] \\
      L_{j_2} \otimes L_{\tilde a_1}^\vee \otimes L_{j'_1} \otimes
      V_{b_0} \ar[r]^m \ar[u]^h \ar[d]^{\mathrm{SP}_L}
      \ar@{}[dr]|{\mathrm{(b)}}
        & L_{j_2} \otimes L_{\tilde a_1}^\vee \otimes V_{a'_1}
        \ar[r]^{\mathrm{SP}_L \otimes \mathrm{SP}_V}
        \ar[d]^{\mathrm{SP}_L} \ar[u]^m
         \ar@{}[ddr]|{\mathrm{(e)}}
          & L_{j_2} \otimes L_{\tilde b_1}^\vee \otimes V_{b'_1}
          \ar[dd]^h \ar[u]^m  \ar@{}[ddr]|{\mathrm{(f)}}&\\
      L_{j_2} \otimes L_{\tilde b_1}^\vee \otimes L_{j'_1} \otimes
      V_{b_0} \ar[r]^m \ar[d]^h  \ar@{}[dr]|{\mathrm{(c)}}
        & L_{j_2} \otimes L_{\tilde b_1}^\vee \otimes V_{a'_1} \ar[d]^h &&\\
      L_{j'_2} \otimes L_{j'_1} \otimes V_{b_0} \ar[r]^m
        & L_{j'_2} \otimes V_{a'_1} \ar[r]^{\mathrm{SP}_V}
          & L_{j'_2} \otimes V_{b'_1} \ar[r]^m
            & V_{a_2}
    }
  \end{equation*}
  Here, each morphism is tensored with an appropriate identity, which
  we leave implicit.  We notice that each of the inner diagrams
  commutes: (a) and (f) due to the compatibility between $h$ and $m$
  (cf.\ \eqref{eq:16}), (b), (c) and (e) because the maps in question
  act independently on the various tensor factors, and (d) due to the
  compatibility between the connections of $L$ and $V$.  Thus the
  outer square commutes, as claimed.
\end{proof}

Finally, we briefly explain how to extend $E$ from $\mathfrak
P(\mathfrak X)$ to the whole of $1\vert1\EBord(\mathfrak X)$.  First,
we require the so-called spin-statistics relation, that is, that the
flip of $1\vert1\EBord(\mathfrak X)$ maps to the grading involution of
$\mathrm{Vect}$.  Second, recall that in Stolz and Teichner's
definition, the bordism category does not admit duals, since a length
zero right elbow $R_0$ is not allowed; this is done because they want
to allow field theories with infinite dimensional state spaces.  In
our example, we could introduce those additional morphisms $R_0$, and
then $E$ would be uniquely determined by the above prescriptions, the
requirement that duals map to duals, and the symmetric monoidal
property.  Concretely, suppose that $\Sigma$ is a family of length
zero left elbows, that is
\begin{equation*}
  \Sigma\colon \bar Y \amalg Y \to \emptyset,
\end{equation*}
where $Y$ denotes an $S$-family of positive superpoints with a choice
of skeleton, say $\iota\colon S \times\mathbb R^{0\vert1} \to U$ and
$\bar Y$ denotes its orientation reversal, with the same underlying
skeleton.  Then we define $E(\bar Y) = V_i^\vee$ and
\begin{equation*}
  E(\Sigma)\colon V^\vee_\iota \otimes V_\iota \to \mathbb C_S
\end{equation*}
to be the evaluation pairing.  The image of other kinds of bordisms is
determined similarly.

\begin{remark}
  A trivialization of $L \to X_1$ allows us to extend, in a fairly
  obvious way, the construction above to $1\vert1\EBord(\mathfrak
  X)^{\mathrm{skel}}$.
\end{remark}

\subsection{Dimensional reduction and the Chern character}
\label{sec:twisted-ch}

In this subsection, we finish the proof of theorem~\ref{thm:4}, by
showing that the diagram indeed commutes.  So our goal is to study the
dimensional reduction of the twisted field theory $E$ associated to
the $\tilde{\mathfrak X}$-twisted vector bundle $\mathfrak V$.  Let us
denote it by
\begin{equation*}
  E' \in 0\vert1\EFT^{T'_{\tilde{\mathfrak X}}}(\Lambda\mathfrak X),
\end{equation*}
and recall, from proposition~\ref{prop:1} and theorem~\ref{thm:3} that
$E'$ determines and is completely determined by an even, closed form
$\omega \in \Omega^*(\Lambda\mathfrak X; L')$.  The underlying form on
the atlas $\hat X_0$, which by abuse of notation we still denote
$\omega$, corresponds to the section
\begin{equation*}
  E'(\mathcal L \circ \check x) \in
  C^\infty(\Pi T\hat X_0; Q')
\end{equation*}
under the isomorphism \eqref{eq:25} and usual identification
$C^\infty(\Pi T\hat X_0) \cong \Omega^*(\hat X_0)$.

\begin{proposition}
  \label{prop:2}
  $\omega = \mathrm{ch}(\mathfrak V) \in \Omega^*(\Lambda\mathfrak X; L')$.
\end{proposition}

\begin{proof}
  It suffices to verify that the underlying forms on $\hat X_0$ agree.
  Our dimensional reduction procedure dictates that $E'(\mathcal L
  \circ \check x) = E(K_{\check x})$, where $K_{\check x} \in
  \mathfrak K(\mathfrak X)$ is the special $\Pi T\hat X_0$-family from
  section~\ref{sec:remind-dimens-reduct}.  By \eqref{eq:4},
  $E(K_{\check x})$ is obtained from the linear map
  \begin{equation*}
    \pi^*V \xrightarrow{\mathrm{SP}^{-1}} \pi^*V
    \xrightarrow{m^{-1}} Q' \otimes \pi^*V,
  \end{equation*}
  by taking supertrace in the $\End(\pi^*V)$ component.  (For clarity,
  we are leaving pullbacks by $p$ and $i$ implicit.)  Here,
  $\mathrm{SP}$ denotes the super parallel transport along the
  superinterval
  \begin{equation*}
    \Pi T\hat X_0 \times [0,1] \subset \Pi T\hat X_0 \times \mathbb
    R^{1\vert1} \xrightarrow{\mathrm{ev}} \hat X_0,
  \end{equation*}
  which \textcite{arXiv:1202.2719} identified with
  $\exp(-\nabla_V^2)$.  Thus, the above homomorphism of vector bundles
  on $\Pi T\hat X_0$ is given by
  \begin{equation*}
    m^{-1}\circ \exp(\nabla_V^2)
    \in C^\infty(\Pi T\hat X_0; \Hom(\pi^*V, Q' \otimes \pi^*V)).
  \end{equation*}
  Using \eqref{eq:25}, our chosen identification $\exp(-\tilde
  B)\colon Q' \to \pi^*L'$, we get
  \begin{equation*}
    \omega = \str(m^{-1} \circ \exp(\nabla_V^2 - B))
    \in \Omega^*(\hat X_0; L'),
  \end{equation*}
  which agrees with the definition \eqref{eq:23} of
  $\mathrm{ch}(\mathfrak V)$.
\end{proof}

\appendix

\section{A primitive integration theory on $\mathbb R^{1|1}$}
\label{sec:fund-theor-calc}

Integration of compactly supported sections of the Berezinian line
bundle is relatively simple to define \cite{MR1701597}.  The notion of
domains with boundary requires more care, as shown by the following
paradox, known as Rudakov's example: on $\mathbb R^{1\vert2}$ with
coordinates $t,\theta_1,\theta_2$, we consider a function $u$ with
$\partial_{\theta_1} u = \partial_{\theta_2} u = 0$.  Then
$\int_{[0,1] \times \mathbb R^{0\vert2}} [dtd\theta]\, u = 0$, but
performing the change of coordinates $t = t' + \theta_1\theta_2$ we
get
\begin{equation*}
  \int_{[0,1] \times \mathbb R^{0\vert2}} [dt'd\theta]\, u(t' + \theta_1\theta_2,\theta_1,\theta_2) =
  \int_{[0,1] \times \mathbb R^{0\vert2}} [dt d\theta]\,
  u + \theta_1\theta_2 \partial_tu = u(1) - u(0).
\end{equation*}
It turns out that the correct notion of boundary of a domain $U$ in a
supermanifold $X$ is a codimension $1\vert0$ submanifold $K
\hookrightarrow X$ whose reduced manifold is the boundary of
$\abs{U}$.  With this proviso, an integration theory featuring the
expected Stokes formula still exists \cite{MR0647158}, and we would
like to describe it concretely in a very special case.

Given $a, b\colon S\to \mathbb R^{1\vert1}$, we define the
superinterval $[b,a] \subset S \times \mathbb R^{1\vert1}$ to be the
domain with boundary prescribed by the embeddings
\begin{gather*}
  i_a\colon S \times \mathbb R^{0\vert1} \hookrightarrow S\times\mathbb
  R^{1\vert1} \xrightarrow{a\cdot} S \times\mathbb R^{1\vert1},\\
  i_b\colon  S \times \mathbb R^{0\vert1} \hookrightarrow S\times\mathbb
  R^{1\vert1} \xrightarrow{b\cdot} S \times\mathbb R^{1\vert1}.
\end{gather*}
We think of $i_a$ as the incoming and $i_b$ as the outgoing boundary
components.  To be consistent with the usual definition of
$1\vert1\EBord$, we will to assume that, modulo nilpotents, $a \geq b$
(cf.\ \textcite[definition~6.41]{MR2648897}).

The fiberwise Berezin integral of a function $u = f + \theta g \in
C^\infty(S\times\mathbb R^{1\vert1})$ on $[b,a]$ will be denoted
$\int_{[b,a]} [dtd\theta]\, u$.  Now, notice that we can always find
primitives with respect to the Euclidean vector field $D =
\partial_\theta-\theta\partial_t$.  In fact, if $G \in C^\infty(S
\times \mathbb R)$ satisfies $\partial_tG = g$, then
\begin{equation*}
  u = D(\theta f - G).
\end{equation*}
It is also clear that any two primitives differ by a constant.  We
have a fundamental theorem of calculus.

\begin{proposition}
  \label{prop:10}
  Given $u, v\in C^\infty(S \times\mathbb R^{1\vert1})$ with $u =
  (\partial_\theta-\theta\partial_t)v$ and $a, b\colon S\to \mathbb
  R^{1\vert1}$, with $a \geq b$ modulo nilpotents, we have
  \begin{equation*}
    \int_{[b,a]}[dtd\theta]\, u = v(b) - v(a).
  \end{equation*}
\end{proposition}

To clarify the meaning of the right-hand side, when using $a,b\colon S
\to \mathbb R^{1\vert1}$, etc., as arguments to a function, we
implicitly identify them with maps $S \to S \times\mathbb
R^{1\vert1}$, to avoid convoluted notation like $v(\mathrm{id}_S,b)$.

\begin{proof}
  Using partitions of unity, it suffices to prove the analogous
  statement for the half-unbounded interval $[b,+\infty]$, namely
  \begin{equation*}
    \int_{[b,+\infty]} [dtd\theta]\, u =  v(b),
  \end{equation*}
  assuming $u$ and $v$ are compactly supported.  Writing $u = f +
  \theta g$, we have $v = \theta f - G$ with $G$ the compactly
  supported primitive of $g$.  Thus,
  \begin{equation*}
    v(b) = b_1f(b_0) - G(b_0),
  \end{equation*}
  where $b_0$, $b_1$ are the components of $b$.  On the other hand the
  embedding $i_b\colon S \times \mathbb R^{0\vert1} \to S\times\mathbb
  R^{1\vert1}$ corresponding to the outgoing boundary of
  $[b,+\infty]$ is expressed, on $T$-points, as
  \begin{equation*}
    (s, \theta) \mapsto (s, b_0+ b_1\theta, \theta + b_1).
  \end{equation*}
  Thus, the domain of integration is picked out by the equation $t\geq
  b_0+ b_1\theta$.  Performing the change of coordinates $t' =
  t-b_0-  b_1\theta$, whose Berezinian is $1$, we get
  \begin{align*}
    \int_{t\geq b_0+ b_1\theta}[dtd\theta]\, u
    &= \int_{t'\geq 0} [dt'd\theta]\, f(t'+b_0+ b_1\theta)
      + \theta g(t'+b_0+ b_1\theta)\\
    &= \int_{t'\geq 0} [dt'd\theta]\,  b_1\theta \partial_tf(t'+b_0)
      + \theta g(t'+b_0)\\
    &= b_1 f(b_0) - G(b_0). \qedhere
  \end{align*}
\end{proof}

As we noticed in the proof, translations on $\mathbb R^{1\vert1}$
preserve the canonical section $[dtd\theta]$ of the Berezinian line;
the flips $\theta \mapsto -\theta$ of course do not.  Thus, an
abstract Euclidean $1\vert1$-manifold $X$ does not come with a
canonical section of $\mathrm{Ber}(\Omega^1_X)$, but the choice of a
Euclidean vector field $D$ fixes a section, which we denote $\vol_D$.
We can then restate the proposition in a coordinate-free way as
follows: for any $S$-family of superintervals $[b,a]$ with a choice of
Euclidean vector field $D$,
\begin{equation*}
\int_{[b,a]} \vol_D Du = u(b)-u(a).
\end{equation*}

\section{Gerbes, twisted vector bundles and Chern forms}
\label{sec:twisted-superconn}

A central extension of the Lie groupoid $X_1 \rightrightarrows X_0$ is
given by
\begin{enumerate}
\item a complex line bundle $L \to X_1$ with connection $\nabla_L$,
\item a (connection-preserving) isomorphism
\begin{math}
  h\colon \pr_2^*L \otimes \pr_1^*L \to c^*L
\end{math}
over the space $X_2 = X_1 \times_{X_0} X_1$ of pairs of composable
morphisms
\item a form $B \in \Omega^2(X_0)$ (called curving).
\end{enumerate}
In more friendly notation, for composable ($S$-points) $f, g \in X_1$,
$h$ is an operation
\begin{equation*}
  h_{f,g}\colon L_f \otimes L_g \to L_{f\circ g}.
\end{equation*}
The multiplication $h$ must satisfy the natural associativity
condition, and the curvature of $L$ the equation $\nabla^2_L = t^*B -
s^*B$.  Note that $dB$ is invariant, and therefore determines a form
$\Omega \in \Omega^3(\mathfrak X)$, called $3$-curvature.  Also, there
are canonical isomorphisms $L\vert_{X_0} \cong \mathbb C$ and
$L_{f^{-1}} \cong L_f^\vee$.  For better legibility, we will typically
use $L_f^\vee$ instead of $L_{f^{-1}}$.

There is an appropriate notion of Morita equivalence for central
extensions \cite[section~4.3]{MR2817778}.  Then, just like
differentiable stacks are Lie groupoids up to Morita equivalence, a
gerbe with connection $\tilde{\mathfrak X}$ over $\mathfrak X$ can be
defined as a central extension up to Morita equivalence.  Gerbes over
an orbifold $\mathfrak X$ are classified by classes in $H^3(\mathfrak
X;\mathbb Z)$, and $[\Omega]$ is the image in de Rham cohomology.

If $L \to X_1 \rightrightarrows X_0$ is a presentation of the gerbe
$\tilde{\mathfrak X}$, then an $\tilde{\mathfrak X}$-twisted vector
bundle $\mathfrak V$ is presented by
\begin{enumerate}
\item a (complex, super) vector bundle $V \to X_0$ with connection
  $\nabla_V$ and
\item an isomorphism $m\colon L \otimes s^*V \to t^*V$ of vector
  bundles with connection over $X_1$ (where the domain is endowed with
  the tensor product connection $\nabla_L \otimes 1 + 1 \otimes
  \nabla_V$)
\end{enumerate}
satisfying certain natural conditions, namely the commutativity of the
following diagrams, where $x,y,z$ and $f\colon x\to y$, $g\colon y\to
z$ denote generic ($S$-)points of $X_0$ respectively $X_1$.
\begin{equation}
  \begin{gathered}
  \label{eq:16}
  \xymatrix@C4em{
    L_g \otimes L_f \otimes V_x \ar[r]^-{\id\otimes m_{f,x}}\ar[d]_{h_{g,f}\otimes\id} 
    & L_g \otimes V_y \ar[d]^{m_{g,y}}\\
        L_{g\circ f} \otimes V_x \ar[r]^-{m_{g\circ f,
        x}}
    & V_z
  }\qquad
  \xymatrix{
    L_{\id_x} \otimes V_x \ar[r]^-{m_{\id_x,x}} \ar[d] & V_x\\
    \mathbb C \otimes V_x \ar[ur]_{\cdot}
  }
  \end{gathered}
\end{equation}

The Chern character form of $\mathfrak V$ is calculated from the above
presentation as follows:
\begin{equation}
  \label{eq:23}
  \mathrm{ch}(\mathfrak V) = \str(i^*m^{-1} \circ p^*\exp(\nabla_V^2 - B))
  \in \Omega^{\mathrm{ev}}(\Lambda\mathfrak X; L').
\end{equation}
(Recall that we write $i\colon \hat X_0 \to X_1$ for the inclusion and
$p\colon \hat X_0 \to X_0$ for the map $s\vert_{\hat X_0} =
t\vert_{\hat X_0}$.)  Let us describe the underlying $L' =
L\vert_{\hat X_0}$-valued differential form on $\hat X_0$ in more
detail.  The isomorphism $m\colon L \otimes s^*V \to t^*V$ gives us an
identification
\begin{equation*}
  \exp(\nabla_L^2) s^*\exp(\nabla_V^2) = t^*\exp(\nabla_V^2).
\end{equation*}
Using the fact that $\nabla_L^2 = t^*B - s^*B$, we get
\begin{equation*}
  s^*(\exp(\nabla_V^2 -B)) = t^*(\exp(\nabla_V^2 -B)),
\end{equation*}
so this defines an $\End(p^*V)$-valued form on $\hat X_0$.  Now,
$i^*m$ is an isomorphism $i^*L \otimes p^*V \to p^*V$, and the form
$\mathrm{ch}(\mathfrak V)$ is obtained by composing the coefficients
of $\exp(\nabla_V^2 - B)$ with
\begin{equation*}
  \End(p^*V) \xrightarrow{i^*m^{-1}} i^*L \otimes \End(p^*V)
  \xrightarrow{\mathrm{id} \otimes \mathrm{str}} i^*L.
\end{equation*}

\begin{proposition}
  This $L\vert_{\hat X_0}$-valued form defines an even, closed element
  in the complex \eqref{eq:20}.
\end{proposition}

\begin{proof}
  This is easy to check directly. It also follows from
  theorem~\ref{thm:3} and proposition~\ref{prop:2}, since we know a
  priori that the form $\omega$ in the statement of that proposition
  is even and closed with respect to the relevant differential.
\end{proof}

\begin{remark}
  If $X_1 \rightrightarrows X_0$ is the groupoid of a finite group
  action on a manifold and the central extension is trivial, then
  $\mathfrak V$ is just the data of an equivariant vector bundle.  In
  this case, $\mathrm{ch}(\mathfrak V)$ represents the equivariant
  Chern character of \textcite{MR928402}.  If $X_1\rightrightarrows
  X_0$ is Morita equivalent to a manifold, then $\mathfrak V$ is what
  is traditionally called a twisted vector bundle, and
  $\mathrm{ch}(\mathfrak V)$ agrees with the definition of
  \textcite{MR1911247}, \textcite{arXiv:1602.02292}, and others.
\end{remark}

\begin{remark}
  Finite-dimensional twisted vector bundles only exist when the
  twisting gerbe represents a torsion class.  Thus, it would be
  interesting to allow a more general target category and investigate
  $1\vert1\EFT$s twisted by non-torsion classes.
\end{remark}

\printbibliography

\end{document}